\documentclass{amsart}

\usepackage{amssymb,amsmath,esint}
\usepackage[all]{xy}
\usepackage{graphicx}
\usepackage{enumerate}

\usepackage{colortbl}
\usepackage{xcolor}

\usepackage{amsrefs}

\colorlet{linkequation}{blue}
\usepackage[colorlinks,plainpages=true,pdfpagelabels,hypertexnames=true,colorlinks=true,pdfstartview=FitV,linkcolor=blue,citecolor=red,urlcolor=black]{hyperref}

\PassOptionsToPackage{unicode}{hyperref}
\PassOptionsToPackage{naturalnames}{hyperref}
\definecolor{dgreen}{rgb}{0,0.5,0}
\definecolor{violet}{rgb}{0.5,0,0.5}
\definecolor{dred}{rgb}{0.7,0,0}
\definecolor{ddred}{rgb}{0.5,0,0}
\definecolor{dblue}{rgb}{0,0,0.5}
\definecolor{ddblue}{rgb}{0,0,0.3}

\newtheorem{theorem}{Theorem}[section]
\newtheorem{lemma}[theorem]{Lemma}
\newtheorem{corollary}[theorem]{Corollary}

\newtheorem{definition}[theorem]{Definition}

\newtheorem{remark}[theorem]{Remark}

\numberwithin{equation}{section}

\DeclareMathOperator*{\ddiv}{div}

\newcommand{ \mr }{ \mathbb{R} }
\newcommand{ \ba }{ \mathbf{a} }
\newcommand{ \m }{ \mathcal{M} }

\newcommand{\integral}[3]{\int_{#1} #2 \ #3}
\newcommand{\integraL}[4]{\int_{#1}^{#2} #3 \ #4}
\newcommand{\mint}[3]{\fint_{#1} #2 \ #3}

\newcommand{\norm}[1]{\left| #1\right|}
\newcommand{\Norm}[1]{\left|\hspace{-0.2mm}\left| #1 \right|\hspace{-0.2mm}\right|}
\newcommand{\gh}[1]{\left( #1\right)}
\newcommand{\mgh}[1]{\left\{ #1\right\}}
\newcommand{\bgh}[1]{\left[ #1\right]}
\newcommand{\vgh}[1]{\left< #1\right>}

\newcommand{\pwz}[1]{C(0,T;L^2(#1)) \cap L^p (0,T; W_0^{1,p}(#1))}

\newcommand{\ik}[1]{K_{#1}^{\lambda}}

\newcommand{\iq}[1]{Q_{#1}^{\lambda}}
\newcommand{\iqq}[1]{Q_{#1}^{\lambda,+}}
\newcommand{\pc}[1]{\left \lfloor {#1} \right \rfloor}

\newcommand{\OO}{\Omega}

\newcommand{\ma}{\mathbf{a}}
\newcommand{\omt}{\Omega_\mathfrak{T}}
\newcommand{\omtb}{\Omega_{\widetilde{T}}}

\newcommand{\al}{\alpha}
\newcommand{\be}{\beta}
\newcommand{\la}{\lambda}
\newcommand{\La}{\Lambda}
\newcommand{\ep}{\varepsilon}
\newcommand{\ka}{\kappa}

\newcommand{\C}{\mathfrak{C}}
\newcommand{\D}{\mathfrak{D}}
\newcommand{\M}{\mathcal{M}}

\newcommand{\bb}{\mathbb{R}}

\AtBeginDocument{%
   \def\MR#1{}
}

\def\XXint#1#2#3{{\setbox0=\hbox{$#1{#2#3}{\int}$}
    \vcenter{\hbox{$#2#3$}}\kern-.5\wd0}}

\begin{document}

\title[Degenerate/singular parabolic equations involving measure data]{Global regularity for degenerate/singular parabolic equations involving measure data}

\author[S.-S. Byun]{Sun-Sig Byun}
\address{S.-S. Byun: Department of Mathematical Sciences and Research Institute of Mathematics, Seoul National University, Seoul 08826, Republic of Korea}
\email{byun@snu.ac.kr}

\author[J.-T. Park]{Jung-Tae Park}
\address{J.-T. Park: Korea Institute for Advanced Study, Seoul 02455, Republic of Korea}
\email{ppark00@kias.re.kr}

%    Information for third author
\author[P. Shin]{Pilsoo Shin}
\address{P. Shin: Department of Mathematics, Kyonggi University, Suwon 16227, Republic of Korea}
\email{shinpilsoo.math@kgu.ac.kr}

\thanks{S.-S. Byun was supported by NRF-2017R1A2B2003877.  J.-T. Park was supported by NRF-2019R1C1C1003844.  P. Shin was supported by NRF-2020R1I1A1A01066850.}

%    General info

%\date{\today}
\date{November 09, 2019}
%\date{March 19, 2019 and, in revised form, July 30, 2019}
%\dedicatory{This paper is dedicated to our authors.}


\begin{abstract}
We consider degenerate and singular parabolic equations with $p$-Laplacian structure in bounded nonsmooth domains when the right-hand side is a signed Radon measure with finite total mass. We develop a new tool that allows global regularity estimates for the spatial gradient of solutions to such parabolic measure data problems, by introducing the (intrinsic) fractional maximal function of a given measure.
\end{abstract}

\maketitle

%%%%%%%%%%%%%%%%%%%%%%%%%%%%%%%%%%%%%%%%%%%%%%%%%%%%%%%%%

%{
%  \hypersetup{linkcolor=black}
%  \tableofcontents
%}

%%%%%%%%%%%%%%%%%%%%%%%%%%%%%%%%%%%%%%%%%%%%%%%%%%%%%%%%%%%%%%%%%%%%%%%%%%%%%%%%%%%%%%%%%%%
\section{Introduction}

We study the following Cauchy-Dirichlet problem for a degenerate/singular parabolic equation:
\begin{equation}
\label{pem1}\left\{
\begin{alignedat}{3}
u_t -\ddiv \mathbf{a}(Du,x,t)  &= \mu &&\quad \text{in}  \  \OO_T, \\
u  &= 0 &&\quad \text{on} \  \partial_p \OO_T.
\end{alignedat}\right.
\end{equation}
Here $\OO_T := \OO \times (0,T)$, $T>0$, is a cylindrical domain with parabolic boundary $\partial_p \OO_T := \gh{\partial \OO \times [0,T]} \cup \gh{\OO \times \{0\}}$, where $\Omega \subset \bb^n$, $n \ge 2$, is a bounded domain with nonsmooth boundary $\partial \Omega$, and we write $Du := D_x u$. The right-hand side term $\mu$ is a signed Radon measure on the domain $\OO_T$
with finite total mass. From now on we assume that the measure $\mu$ is defined on $\bb^{n+1}$ by letting zero outside $\OO_T$; that is,
\begin{equation*}
|\mu|(\OO_T) = |\mu|(\bb^{n+1})< \infty.
\end{equation*}
 The nonlinear operator $\mathbf{a}=\mathbf{a}(\xi,x,t):
\bb^n \times \bb^n \times \bb \rightarrow \bb^n$ is assumed to be
measurable in $x$- and $t$-variables and satisfies the following growth and ellipticity
conditions:
\begin{equation}\label{str1}\left\{
\begin{aligned}
& |\mathbf{a}(\xi,x,t)| + |\xi||D_{\xi}\mathbf{a}(\xi,x,t)| \le \Lambda_1 |\xi|^{p-1},\\
& \Lambda_0 |\xi|^{p-2}|\eta|^2 \le \left< D_{\xi}\mathbf{a}(\xi,x,t)\eta,\eta \right>,
\end{aligned}\right.
\end{equation}
for almost every $(x, t) \in \mathbb{R}^n \times \bb$, for every $\eta \in \mathbb{R}^n$, $\xi \in \bb^n
\setminus \{0\}$ and for some constants $\Lambda_1 \ge \Lambda_0 >0$.  Here $D_{\xi}\mathbf{a}(\xi,x,t)$ is the Jacobian matrix of the operator $\mathbf{a}$ with respect to the variable $\xi$,
and $\vgh{\cdot, \cdot}$ is the standard inner product in $\bb^n \times \bb^n$. This type of study is modeled after the parabolic $p$-Laplace equation
$$u_t - \ddiv\gh{|Du|^{p-2}Du}  = \mu,$$
as one can take $\mathbf{a}(\xi,x,t) =|\xi|^{p-2}\xi$ in the problem \eqref{pem1}.
In this paper we assume
\begin{equation}\label{p-range}
p>2-\frac{1}{n+1},
\end{equation}
which ensures that the spatial gradient of our solution belongs to $L^1(\OO_T)$ (see Section \ref{subsola} below for more details).
Note that the structure conditions \eqref{str1} imply $\mathbf{a}(0,x,t)=0$ for $(x,t) \in \bb^n \times \bb$ and the following monotonicity condition:
\begin{equation}\label{monotonicity}
\hspace{-2mm} \left< \mathbf{a}(\xi_1,x,t)-\mathbf{a}(\xi_2,x,t), \xi_1-\xi_2 \right> \ge
\left\{
\begin{alignedat}{3}
&\tilde{\La}_0 \norm{\xi_1-\xi_2}^{p} && \ \ \text{if}  \ p \ge 2,\\
&\tilde{\La}_0 \left( \norm{\xi_1}^2 + \norm{\xi_2}^2
\right)^{\frac{p-2}{2}} \norm{\xi_1-\xi_2}^2 && \ \ \text{if}  \ 1<p<2
\end{alignedat}\right.
\end{equation}
for all $(x, t) \in \mathbb{R}^n \times \bb$ and $\xi_1, \xi_2 \in \bb^n$, and for some constant
$\tilde{\La}_0=\tilde{\La}_0(n,\La_0,p)>0$.

Measure data problems including the problem \eqref{pem1} arise in a variety of models, primarily in physics and biology. For instance, the incompressible Navier-Stokes equations with measure data describe in biology the flow pattern of blood in the heart \cite{Pes77,PM89}. This relates to the design of artificial hearts. Measure data problems appear as well in the study of surface tension forces concentrated on the interfaces of fluids \cite{SSO94, LL94, OF03} and state-constrained optimal control theory \cite{Cas86, Cas93, CdT08, MPS11}.

Concerning regularity results for the parabolic problem \eqref{pem1}, the first contribution has been given by Boccardo and Gallou\"{e}t \cite{BG89} (see Section \ref{subsola} below for more detailed explanation). Later this results lead to sharp estimates in terms of Marcinkiewicz spaces, see \cite{Bar14, BD18} for $p\ge2$ and \cite{Bar17} for $2-\frac{1}{n+1}<p<2$. For pointwise potential estimates for the spatial gradient of solutions, finer regularity results have been proven in \cite{DM11} for $p=2$, \cite{KM14b,KM14c} for $p\ge2$, and \cite{KM13b} for $2-\frac{1}{n+1}<p\le2$ (see also survey papers \cite{Min11b,KM14a} and the references given there for an overview of potential estimates). Lastly, Calder\'on-Zygmund  type estimates, namely the spatial gradient estimates of solutions corresponding to given data, have been obtained in \cite{Ngu15, BP18b,BH12} only for $p=2$ in the literature of parabolic measure data problems.

The main goal of this article is to prove global Calder\'on-Zygmund type estimates for the problem \eqref{pem1} under \eqref{str1} and \eqref{p-range} (see Theorem \ref{main theorem} and \ref{main theorem2} below). Here our proof covers both the cases of $p\ge2$ and of $2-\frac{1}{n+1} <p\le2$. For the regularity estimates, we need to consider a suitable notion of solution (Definition \ref{sola}) as well as find optimal regularity requirements on the operator $\mathbf{a}$ and the boundary $\partial\OO$ (Definition \ref{main assumption}). We also refer to \cite{Min07,Min10, Phu14a} for the Calder\'on-Zygmund type estimates to the elliptic case of \eqref{pem1}.

Before discussing how to prove Theorem \ref{main theorem} and \ref{main theorem2}, we introduce an effective and systematic approach to obtain Calder\'on-Zygmund type estimates for quasilinear elliptic and parabolic problems having divergence form: the so-called {\em maximal function technique}. This technique by Caffarelli and Peral \cite{CP98} uses the standard energy estimate, Hardy-Littlewood maximal function and Calder\'on-Zygmund decomposition. The advantage of this approach lies in the fact that it can avoid the use of singular integrals and explicit kernels. It works for nonlinear elliptic equations and nondegenerate ($p=2$) parabolic equations (see for instance \cite{Min10,Phu14a,Ngu15,BP18b}). However, this method does not work for degenerate/singular ($p\neq2$) parabolic equations. This is not only because of the structure of parabolic $p$-Laplace type equations (roughly speaking it scales differently in space and time), but also because of the absence of suitable maximal function operators associated with such parabolic equations.

To overcome these difficulties (as mentioned above) and to derive our main estimates (Theorem \ref{main theorem} and \ref{main theorem2}), we define suitable intrinsic (fractional) maximal function operators (see Section \ref{notation} and \ref{Preliminary tools} below for definitions and some properties) and verify the precise relationship between the intrinsic (fractional) maximal functions and the classical ones (see Section \ref{la covering arguments} and \ref{Global gradient estimates for parabolic measure data problems} below). Our approach proposed here would be technically delicate, but it eventually provides a powerful and useful assertion concerning the desired regularity estimates for the problem \eqref{pem1}.
It is also worth pointing out that Acerbi and Mingione \cite{AM07} have obtained Calder\'on-Zygmund type estimates for parabolic problems having $p$-Laplacian structure (without measure data), by using the so-called {\em maximal function free technique}. This approach is based upon a stopping-time argument and direct PDE estimates without using maximal functions (see for example \cite{Bar14,Bar17,Bog14,BD18,BOR13} for regularity results via this method).

Let us outline the plan of this paper as follows: in Section \ref{Preliminaries and main results}, we introduce some notion, preliminary tools, a suitable solution and regularity assumptions to establish our main theorems (Theorem \ref{main theorem} and \ref{main theorem2}). Section \ref{Standard energy type estimates} concerns the standard energy type estimates for the problem \eqref{pem1}. In Section \ref{intrinsic comparison}, we investigate comparison estimates between the problem \eqref{pem1} and its limiting problems. Section \ref{la covering arguments}, the main ingredient of this paper,  is devoted to verifying the assumptions of the covering argument (Lemma \ref{VitCov}) under the intrinsic parabolic cylinders alongside our intrinsic (fractional) maximal operator. Finally, in Section \ref{Global gradient estimates for parabolic measure data problems}, we prove the global Calder\'on-Zygmund type estimates (Theorem \ref{main theorem} and \ref{main theorem2}) for the problem \eqref{pem1}.

%%%%%%%%%%%%%%%%%%%%%%%%%%%%%%%%%%%%%%%%%%%%%%%%%%%%%%
\section{Preliminaries and main results}\label{Preliminaries and main results}

\subsection{Notation}\label{notation}
Let us first introduce general notation, which will be used later. A point $x \in \bb^n$ will be written $x = (x_1, \cdots, x_n)$. Let $B_r(x_0)$ denote the open ball in $\bb^n$ with center $x_0$ and radius $r>0$, and let $B_r^+(x_0) := B_r(x_0) \cap \{x \in \bb^n : x_n >0 \}$. The standard parabolic cylinder in $\bb^n \times \bb =: \bb^{n+1}$ is denoted $Q_r(x_0,t_0):=B_r(x_0) \times (t_0-r^2,t_0+r^2)$ with center $(x_0,t_0) \in \bb^{n+1}$, radius $r$ and height $2r^2$. With $\lambda > 0$, we will often consider the {\em intrinsic parabolic cylinder}
\begin{equation}\label{ipc}
\iq{r}(x_0,t_0) := B_r(x_0) \times (t_0- \lambda^{2-p} r^2, t_0+ \lambda^{2-p} r^2),
\end{equation}
see \cite{DiB93,Urb08,KM14b} for the concept of intrinsic geometry related to the intrinsic parabolic cylinder.
We also use the shorter notation as follows:
\begin{gather*}
\OO_T := \OO \times (0,T), \ \omt := \OO \times (-\infty,T), \ \omtb := \OO \times (-T,T),\\
\ik{r}(x_0,t_0) := \iq{r}(x_0,t_0) \cap \omt,\\
%I_r(t_0) := (t_0-r^2,t_0+r^2), \ I_r^\la(t_0) :=(t_0- \lambda^{2-p} r^2, t_0+ \lambda^{2-p} r^2),\\
%\OO_r(x_0) := B_r(x_0) \cap \OO, \ K_r(x_0,t_0) := Q_r(x_0,t_0) \cap \OO_T, \ \ik{r}(x_0,t_0) := \iq{r}(x_0,t_0) \cap \OO_T,\\
\iqq{r}(x_0,t_0) := B_r^+(x_0) \times (t_0- \lambda^{2-p} r^2, t_0+ \lambda^{2-p} r^2),\\
T_r^\lambda(x_0,t_0) := \gh{B_r(x_0) \cap \{x \in \bb^n : x_n =0 \}} \times (t_0- \lambda^{2-p} r^2, t_0+ \lambda^{2-p} r^2).
\end{gather*}
%and for simplicity if $(x_0,t_0):=(0,0)$ we write $B_r := B_r(0)$, $Q_r := Q_r(0,0)$, $\iq{r} := \iq{r}(0,0)$, $\ik{r} := \ik{r}(0,0)$, and so on.
Given a real-valued function $f$, we write
\begin{equation*}
(f)_+ := \max\mgh{f,0} \quad \text{and} \quad (f)_- := - \min\mgh{f,0}.
\end{equation*}
For each set $U \subset \bb^{n+1}$, $|U|$ is the $(n+1)$-dimensional Lebesgue measure of $U$, $diam(U)$ is the diameter of $U$, and $\chi_U$ is the usual characteristic function of $U$.
For $f \in L_{loc}^1(\bb^{n+1})$, $\bar{f}_{U}$ stands for the integral average of $f$ over a bounded open set $U \subset \bb^{n+1}$; that is,
$$\bar{f}_{U} := \mint{U}{f(x,t)}{dx dt} := \frac{1}{|U|} \integral{U}{f(x,t)}{dx dt}.$$
We use both the notation $f_t$ and $\partial_t f$ to denote the time derivative of a function $f$. In what follows, we denote by $c$ to mean any constant that can be computed in terms of known quantities; the exact value denoted by $c$ may be different from line to line.

Now we introduce the fractional maximal function used in this paper. For a deeper discussion of the fractional maximal function, we refer the reader to \cite{AH96,KS03,Min10,Phu14a}. For a locally finite measure $\nu$ on $\bb^{n+1}$,  the {\em fractional maximal function of order $1$ for $\nu$}, denoted by $\M_1(\nu)$, is defined as
\begin{equation}\label{fractional mf}
\m_1(\nu) (x,t) := \sup_{r>0} \frac{ |\nu|(Q_r(x,t))}{r^{n+1}} \quad \text{for} \ \ (x,t) \in \bb^{n+1}.
\end{equation}
When a measure $\nu_0$ is given on $\bb^n$, we similarly define $\M_1(\nu_0)$, with the standard parabolic cylinder $Q_r(x,t)$ replaced by the ball $B_r(x)$; that is,
\begin{equation}\label{fractional mf2}
\m_1(\nu_0) (x) := \sup_{r>0} \frac{ |\nu_0|(B_r(x))}{r^{n-1}} \quad \text{for} \ \ x \in \bb^n.
\end{equation}
 Also, with $\lambda >0$, we define the {\em intrinsic fractional maximal function of order $1$ for $\nu$} to be
\begin{equation}\label{intrinsic fractional mf}
\m_1^\lambda(\nu) (x,t) := \sup_{r>0} \frac{ |\nu|(\iq{r}(x,t))}{r^{n+1}} \quad \text{for} \ \ (x,t) \in \bb^{n+1}.
\end{equation}

%%%%%%%%%%%%%
\subsection{SOLA (Solution Obtained by Limits of Approximations)}\label{subsola}
Let us first introduce a weak solution to the problem \eqref{pem1} when $\mu$ on the right-hand side of \eqref{pem1} is some integrable function. We say $u \in \pwz{\OO}$ is a {\em weak solution} of \eqref{pem1} if
\begin{equation}\label{weak sol}
\integral{\OO_T}{-u\varphi_t + \vgh{\mathbf{a}(Du,x,t), D\varphi}}{dxdt} = \integral{\OO_T}{\mu\varphi}{dxdt}
\end{equation}
for all testing functions $\varphi \in C_c^\infty(\OO_T)$. With the concept of the Steklov average (see Section \ref{Preliminary tools} below for definition), we can also define the following discrete version of a weak formulation, which is equivalent to \eqref{weak sol}:
\begin{equation}\label{steklov weak sol}
\integral{\OO\times\{t\}}{\partial_t [u]_h \varphi + \vgh{\bgh{\mathbf{a}(Du,x,t)}_h, D\varphi }}{dx} = \integral{\OO\times\{t\}}{[\mu]_h\varphi}{dx}
\end{equation}
for all $0 < t < T-h$ and all $\varphi \in W_0^{1,p}(\OO)$.
However, since the right-hand side measure $\mu$ of \eqref{pem1} does not in general belong to the dual space of $\pwz{\OO}$, a solution $u$ of \eqref{pem1} does not become a weak solution. For instance, let us consider the parabolic $p$-Laplace problem
\begin{equation}\label{baren}
u_t - \ddiv \gh{|Du|^{p-2} Du} = \delta_0 \quad \text{in} \hspace{2mm} \bb^n \times(0,\infty),
\end{equation}
where $\delta_0$ is the Dirac measure charging the origin. Then the {\em fundamental solution} of \eqref{baren} is
\begin{equation*}\label{baren sol}
\Gamma(x,t) = \left\{
\begin{alignedat}{3}
&t^{-n\theta} \gh{c_b - \frac{p-2}{p} \theta^{\frac{1}{p-1}} \gh{\frac{|x|}{t^\theta}}^{\frac{p}{p-1}} }_{+}^{\frac{p-1}{p-2}}&&\quad \text{if} \ \ p\neq 2,\\
&(4\pi t)^{-\frac{n}{2}} e^{-\frac{|x|^2}{4t}}&&\quad \text{if} \ \ p= 2,
\end{alignedat}\right.
\end{equation*}
for some  $c_b = c_b(n,p) >0$. Here $\theta := \frac{1}{p(n+1) - 2n}$. The solution $\Gamma$ is well defined provided that $\theta >0$; that is, $p > \frac{2n}{n+1}$. %{\color{dred}\bf Let us now show that the spatial gradient of the solution $\Gamma$ of \eqref{baren} belongs to $L^q(0,\infty; L^q(\bb^n))$, using some dimension analysis. ???}
In addition, we see from the similarity method (see  for example \cite[Chapter 4.2.2]{Eva10}) that $[t]^\theta = [|x|]$, $[\tau]$ denoting the dimension of the quantity $\tau$. %This means that the dimension of the integrability of the solution $\Gamma$ has same  to $t^{-n\theta}$.
Then we deduce
\begin{equation}\label{baren sol range}
\int_{0}^{\infty}\integral{\bb^n}{|D\Gamma|^q}{dxdt} < \infty \quad \text{for all} \ \ q < p - \frac{n}{n+1},
\end{equation}
which implies the solution $\Gamma \notin L^p \gh{0,\infty;W^{1,p}(\bb^n)}$; that is, $\Gamma$ is not a weak solution. The result in \eqref{baren sol range} also gives us that $\Gamma \in L^1 \gh{0,\infty;W^{1,1}(\bb^n)}$ if and only if
\begin{equation*}\label{baren sol range2}
p- \frac{n}{n+1} > 1 \iff p > 2 - \frac{1}{n+1} \gh{> \frac{2n}{n+1}}.
\end{equation*}
This relation shows that the condition \eqref{p-range} is natural and crucial.

For this reason, we need to consider a more general class of solutions beyond the notion of weak solutions.
\begin{definition}
\label{sola} $u \in L^1(0,T; W_0^{1,1}(\OO))$ is a {\em SOLA} to \eqref{pem1} if there exists a sequence of weak solutions $\{u_m\}_{m \geq 1} \subset \pwz{\OO}$
of the regularized problems
\begin{equation}\label{sola eq}\left\{
\begin{alignedat}{3}
\partial_t u_m -\ddiv  \mathbf{a}(Du_m,x,t) &= \mu_m &&\quad \text{in} \ \OO_T,\\
u_m &= 0 &&\quad \text{on} \ \partial_p \OO_T
\end{alignedat}\right.
\end{equation}
such that
\begin{equation}\label{sola1}
u_m \to u \quad \text{in} \ \ L^{\tilde{d}}\gh{0,T; W_0^{1,\tilde{d}}(\OO)} \ \
\textrm{   as } \ m \rightarrow \infty,
\end{equation}
where $\tilde{d}:=\max\{1,p-1\}$. Here $\mu_m \in L^{\infty}(\OO_T)$ converges weakly to $\mu$ in the
sense of measure and satisfies that for each cylinder $Q := B \times (t_1,t_2) \subset
\bb^{n+1}$,
\begin{equation}\label{sola2}
\limsup_{m \to \infty} |\mu_m|(Q) \le |\mu|(\pc{Q}),
\end{equation}
with $\mu_m$ defined in $\bb^{n+1}$ by the zero extension of $\mu_m$ to $\bb^{n+1}
\setminus \OO_T$.
\end{definition}

In the right-hand side of \eqref{sola2}, the symbol $\pc{Q}$ denotes the parabolic closure of $Q$ defined as
\begin{equation}\label{parabolic closure}
\pc{Q} := Q \cup \partial_p Q,
\end{equation}
where $\partial_p Q := \gh{\partial B \times [t_1,t_2]} \cup \gh{B \times \{t_1\}}$ is the parabolic boundary of $Q$. We regard $\mu_m$ as a suitable convolution of the measure $\mu$ via mollification (see for instance \cite[Lemma 5.1]{KLP10}). Then we obtain $\mu_m \rightharpoonup \mu$ in the sense of measure satisfying the property \eqref{sola2} and the following uniform $L^1$-estimate:
\begin{equation}\label{sola3}
\Norm{\mu_m}_{L^1(\OO_T)} \le |\mu|(\OO_T).
\end{equation}

An approximation method in \cite{BG89, BDGO97} ensures that there exists a SOLA $u$ of the problem \eqref{pem1} such that
\begin{equation*}\label{SOLA range}
u_m \to u \quad \text{in} \ \ L^q(0,T; W_0^{1,q}(\OO)) \quad \text{for all} \ \ q < p- \frac{n}{n+1}.
\end{equation*}
We note that if the measure $\mu$ belongs to the dual space $L^{p'}(0,T;W^{-1,p'}(\OO))$, then a SOLA and a weak solution actually coincide each other. On the other hand, the uniqueness of SOLA still remains unsolved except for $\mu \in L^1(\OO_T)$ or the linear case ($p=2$); that is, $\mathbf{a}(\xi,x,t) =\mathbf{a}(x,t)\xi$. We refer to \cite{BG89, BDGO97,Pet08} and the references therein for a further discussion.

%%%%%%%%%%%%%%%%%%%%%%%%%%%%%%%%%%%%%%%%%
\subsection{Main results}\label{Main theorems}
Let us first introduce the main regularity assumptions on the operator $\mathbf a$ and the boundary of $\Omega$.

\begin{definition}\label{main assumption}
Given $R>0$ and $\delta \in \gh{0,\frac{1}{8}}$, we say $(\mathbf{a},\OO)$ is {\em $(\delta,R)$-vanishing} if the following two conditions hold:
\begin{enumerate}
\item\label{2-small BMO} The operator $\mathbf{a}(\xi,x,t)$ satisfies
\begin{equation*}\label{2-small BMO}
\sup_{t_1,t_2 \in \bb} \sup_{0<r \le R} \sup_{y\in\bb^n} \fint_{t_1}^{t_2} \mint{B_r(y)}{\Theta\gh{\mathbf{a},B_r(y)}(x,t)}{dx dt} \le \delta,
\end{equation*}
where
\begin{equation*}\label{2-AA}
\Theta\gh{\mathbf{a},B_r(y)}(x,t) := \sup_{\xi \in \bb^n \setminus \{0\}} \frac{\left|\mathbf{a}(\xi,x,t) - \bar{\mathbf{a}}_{B_r(y)}(\xi,t)  \right|}{|\xi|^{p-1}}
\end{equation*}
with $\bar{\mathbf{a}}_{B_r(y)}(\xi,t)$ denoting the integral average of $\mathbf{a}(\xi,\cdot,t)$ over the ball $B_r(y)$; that is,
\begin{equation*}
\bar{\mathbf{a}}_{B_r(y)}(\xi,t) := \mint{B_r(y)}{\mathbf{a}(\xi,x,t)}{dx}.
\end{equation*}
\item The domain $\OO$ is called {\em $(\delta,R)$-Reifenberg flat} if
for each $y_0 \in \partial\OO$ and each $r \in (0,R]$, there exists
a new coordinate system $\{y_1,\cdots,y_n\}$ such that in this
coordinate system, the origin is $y_0$ and
\begin{equation*}
\label{Reifenberg} B_r(0) \cap \mgh{y\in\bb^n : y_n > \delta r} \subset B_r(0) \cap \OO \subset B_r(0) \cap \mgh{y\in\bb^n : y_n > - \delta r}.
\end{equation*}
\end{enumerate}
\end{definition}

\begin{remark}\label{main remark}
\begin{enumerate}[(i)]
\item The parameter $\delta$ is sufficiently small to be determined in the proofs
of Theorem \ref{main theorem} and \ref{main theorem2}. This number is invariant under the
scaling of the problem (\ref{pem1}), and the number $R$ is given arbitrary.

\item The condition \eqref{2-small BMO} of Definition \ref{main assumption} implies that $\mathbf{a}(\xi,x,t )$ is just measurable in the $t$-variable and of {\em small BMO (Bounded Mean Oscillation)} in the $x$-variable uniformly in $\xi$.

\item The boundary of the $(\delta,R)$-Reifenberg flat domain $\OO$ can be trapped between two hyperplanes at the small scales. This boundary includes Lipschitz boundary with a small Lipschitz constant.
Moreover, if $\OO$ is $(\delta,R)$-Reifenberg flat, then we see that the
following measure density conditions hold:
\begin{equation}\label{measure density}\left\{
\begin{aligned}
& \sup_{0<r \le R} \sup_{y \in \OO} \frac{|B_r(y)|}{|\OO \cap B_r(y)|} \le \gh{\frac{2}{1-\delta}}^n \le \gh{\frac{16}{7}}^n,\\
& \inf_{0<r \le R} \inf_{y \in \partial\OO} \frac{|\OO^c \cap B_r(y)|}{|B_r(y)|} \ge \gh{\frac{1-\delta}{2}}^n \ge \gh{\frac{7}{16}}^n.
\end{aligned}\right.
\end{equation}
For a further discussion on Reifenberg flat domains, we refer to \cite{Tor97,LMS14} and the references therein.
\end{enumerate}
\end{remark}

We are ready to present the following global Caldr\'{o}n-Zygmund type estimates for our problem \eqref{pem1}:
\begin{theorem}\label{main theorem}
Let $p>2-\frac{1}{n+1}$ and let $u$ be a SOLA of the problem \eqref{pem1}. Then for any $\max\{1,p-1\} < q <\infty$, there exists a small constant $\delta=\delta(n,\La_0,\La_1,p,q) >0$ such that if $(\mathbf{a},\OO)$ is $(\delta,R)$-vanishing for some $R>0$, then we have the following estimates:
\begin{enumerate}
\item for $p\ge2$,
\begin{equation}\label{main r1-1}
\integral{\OO_T}{|Du|^q}{dx dt} \le c \mgh{ \integral{\OO_T}{\bgh{\M_1(\mu)}^{q}}{dx dt} + \bgh{|\mu|(\OO_T)}^{\frac{(n+2)(p-1)q}{(n+1)p-n}}  +1}
\end{equation}
for some constant $c = c(n,\La_0,\La_1,p,q,R,\OO_T) \ge 1$, where $\M_1(\mu)$ is given in \eqref{fractional mf};

\item for $2-\frac{1}{n+1} < p \le 2$,
\begin{equation}\label{main r1-2}
\integral{\OO_T}{|Du|^q}{dx dt} \le c \mgh{ \int_{\OO_T} \bgh{\M_1(\mu)}^{\frac{2q}{(n+1)p-2n}} dx dt + \bgh{|\mu|(\OO_T)}^{\frac{2q}{(n+1)p-2n}}  +1}
\end{equation}
for some constant $c = c(n,\La_0,\La_1,p,q,R,\OO_T) \ge 1$.
\end{enumerate}
\end{theorem}

\begin{remark}\label{main rk1}
\begin{enumerate}[(i)]
\item Note that both the constants $c$ in \eqref{main r1-1} and \eqref{main r1-2} are stable as $p \to 2$. On the other hand, they blow up as $q \searrow \max\{1,p-1\}$, see Remark \ref{de remark} below.

\item For $0<q\le \max\{1,p-1\}$, gradient estimates like \eqref{main r1-1} and \eqref{main r1-2} directly follow from \eqref{mest} and Lemma \ref{standard estimate} below. In this case, we also refer to \cite{Bar14, Bar17, BD18} in the setting of  more general function frames including Marcinkiewicz spaces.
\end{enumerate}
\end{remark}

\begin{remark}\label{main rk2}
Note that $|\mu|(\OO_T) \le diam(\OO_T)^{n+1} \M_1(\mu)$, which implies
\begin{equation}\label{mest}
\bgh{|\mu|(\OO_T)}^{\alpha q} \le c(n,\alpha,q,\OO_T) \gh{\integral{\OO_T}{\M_1(\mu)}{dxdt}}^{\alpha q}
\end{equation}
for any $\alpha>0$.
\begin{enumerate}[(i)]
\item If $p=2$, then $\frac{(n+2)(p-1)}{(n+1)p-n} = \frac{2}{(n+1)p-2n} =1$; it follows from Theorem \ref{main theorem} and \eqref{mest} that
\begin{equation*}
\integral{\OO_T}{|Du|^q}{dx dt} \le c \mgh{ \integral{\OO_T}{\bgh{\M_1(\mu)}^q}{dx dt}+1},
\end{equation*}
which is the main estimate in \cite{BP18b}.

\item If $p > 2$, then $\frac{(n+2)(p-1)}{(n+1)p-n} > 1$; we derive from \eqref{main r1-1} and \eqref{mest} that
\begin{equation}\label{main r1-3}
\integral{\OO_T}{|Du|^q}{dx dt} \le c \mgh{\gh{ \integral{\OO_T}{\bgh{\M_1(\mu)}^q}{dx dt} }^{\frac{(n+2)(p-1)}{(n+1)p-n}} + 1}.
\end{equation}

\item If $2- \frac{1}{n+1} < p < 2$, then $\frac{2}{(n+1)p-2n}>1$; we have from \eqref{main r1-2} and \eqref{mest} that
\begin{equation}\label{main r1-4}
\integral{\OO_T}{|Du|^q}{dx dt} \le c \mgh{ \integral{\OO_T}{\bgh{\M_1(\mu)}^{\frac{2q}{(n+1)p-2n}}}{dx dt} +1}.
\end{equation}
\end{enumerate}
\end{remark}

\begin{remark}\label{main rk3}
The occurrence of the exponents $\frac{(n+2)(p-1)}{(n+1)p-n}$ and $\frac{2}{(n+1)p-2n}$ in Theorem \ref{main theorem} is closely related to the structure and anisotropic property (a constant multiple of solution no longer becomes another solution) of the problem \eqref{pem1}. More precisely, the exponent  $\frac{(n+2)(p-1)}{(n+1)p-n}$ comes from the standard energy type estimate for the problem \eqref{pem1}, see Lemma \ref{standard estimate} and \ref{DecayLem} below. On the other hand, the exponent $\frac{2}{(n+1)p-2n}$ appears owing to the ratio between the standard parabolic cylinder $Q_r(x,t)$ and the intrinsic parabolic cylinder $\iq{r}(x,t)$, see Lemma \ref{difference of ratio} below.

Also, there are similar situations in the regularity theory for parabolic $p$-Laplace type problems. We refer to \cite{KM13b,KM14b,KM14c} for potential estimates and \cite{AM07,BOR13,Bog14} for Calder\'{o}n-Zygmund estimates, respectively.
\end{remark}

If the measure $\mu$ is time-independent or admits a favorable decomposition, see \eqref{measure decomposition} below, then we obtain more sharp gradient estimates than the estimates in Theorem \ref{main theorem} as follows:
\begin{theorem}\label{main theorem2}
Let $p>2-\frac{1}{n+1}$ and let $\tilde{d}< q <\infty$, where $\tilde{d}:=\max\{1,p-1\}$. Suppose that the following decomposition holds:
\begin{equation}\label{measure decomposition}
\mu=\mu_0 \otimes f,
\end{equation}
where $\mu_0$ is a finite signed Radon measure on $\OO$ and $f \in L^{\frac{q}{p-1}}(0,T)$. Then there exists a small constant $\delta=\delta(n,\La_0,\La_1,p,q) >0$ such that the following holds: if $(\mathbf{a},\OO)$ is $(\delta,R)$-vanishing for some $R>0$, then for any SOLA $u$ of the problem \eqref{pem1} we have
\begin{equation}\label{main r2}
\begin{aligned}
\integral{\OO_T}{|Du|^q}{dx dt} \le c &\left\{\integral{\OO_T}{\left[\left(\M_1(\mu_0)\right)f\right]^\frac{q}{p-1}}{dx dt}\right.\\
&\qquad \qquad + \left.\bgh{|\mu_0|(\OO)\|f\|_{L^1(0,T)}}^{\frac{(n+2)\tilde{d}q}{(n+1)p-n}} + 1\right\}
\end{aligned}
\end{equation}
for some constant $c = c(n,\La_0,\La_1,p,q,R,\OO_T) \ge 1$, where $\M_1(\mu_0)$ is given in \eqref{fractional mf2}.% and
%\begin{equation}\label{tilded}
%\tilde{d} := \left\{
%\begin{alignedat}{2}
%&p-1 &&\quad \text{if} \ \  p\ge2, \\
%&1 &&\quad \text{if} \ \ 2-\frac{1}{n+1}< p \le 2.
%\end{alignedat}\right.
%\end{equation}
\end{theorem}

\begin{remark}\label{main rk4}
\begin{enumerate}[(i)]
\item  The constant $c$ in \eqref{main r2} blows up as $q \searrow \max\{1,p-1\}$, see Remark \ref{de remark} and \eqref{61} below.

\item For $p>2-\frac{1}{n+1}$, we see from \eqref{main r2} and \eqref{mest} that
\begin{equation}\label{main r2-2}
\integral{\OO_T}{|Du|^q}{dx dt} \le c \left\{
\begin{alignedat}{3}
&\gh{ \integral{\OO_T}{\bgh{\left(\M_1(\mu_0)\right)f}^{\frac{q}{p-1}}}{dx dt} }^{\frac{(n+2)(p-1)^2}{(n+1)p-n}} + 1 &&\quad \text{if} \ \ p \ge 2,\\
&\integral{\OO_T}{\bgh{\left(\M_1(\mu_0)\right)f}^{\frac{q}{p-1}}}{dx dt} + 1 &&\quad \text{if} \ \ p \le 2.
\end{alignedat}\right.
\end{equation}
Comparing \eqref{main r1-3}--\eqref{main r1-4} with \eqref{main r2-2}, we observe that the estimate \eqref{main r2-2} gives a more natural and sharp result.
\end{enumerate}
\end{remark}

%%%%%%%%%%%%%%%%%%%%%%%%%%%
\subsection{Preliminary tools}\label{Preliminary tools}
Let us first recall definition and some properties of  Steklov average, see \cite{DiB93} for details. Given  $0<h<T$, the {\em Steklov average $[f]_h$} of a function $f \in L^1(\OO_T)$ is defined by
\begin{equation*}\label{steklov}
[f]_h(\cdot,t):= \left\{
\begin{alignedat}{3}
&\frac1h \integraL{t}{t+h}{f(\cdot,\tau)}{d\tau} &&\quad \text{for} \ \ t \in (0, T-h],\\
&0&&\quad \text{for} \ \ t  > T-h.
\end{alignedat}\right.
\end{equation*}

\begin{lemma}[See {\cite[Chapter I, Lemma 3.2]{DiB93}}]
\label{steklov property}
Let $f \in L^r(0,T;L^q(\OO))$ for some $q,r \ge1$. Then $[f]_h \to f$ in $L^r(0,T-\varepsilon;L^q(\OO))$ as $h \to 0$ for every $\varepsilon \in (0,T)$. If $f \in C(0,T;L^q(\OO))$, then $[f]_h(\cdot,t) \to f(\cdot,t)$ in $L^q(\OO)$ for every $t \in (0,T-\varepsilon)$ and every $\varepsilon \in (0,T)$.
\end{lemma}

We also use a {\em parabolic embedding theorem} as follows:
\begin{lemma}[See {\cite[Chapter I, Proposition 3.1]{DiB93}}]
\label{parabolic embedding}
Let $q,l \ge1$. Then there is a constant $c=c(n,q,l)\ge1$ such that for every $f \in L^\infty(0,T;L^l(\OO)) \cap L^q(0,T;W_0^{1,q}(\OO))$, we have
\begin{equation*}
\integral{\OO_T}{|f|^{q\frac{n+l}{n}}}{dxdt} \le c\gh{\integral{\OO_T}{|Df|^q}{dxdt}}\gh{\sup_{0<t<T}\integral{\OO\times\{t\}}{|f|^l}{dx}}^{\frac{q}{n}}.
\end{equation*}

\end{lemma}

We now introduce some analytic and geometric properties which will be crucially used in this paper. Let $f$ be a locally integrable function in $\mr^{n+1}$ and let $\lambda>0$. We define the {\em (intrinsic) $\lambda$-maximal function} of $f$ as
\begin{equation*}
\m^\lambda f (x,t) := \sup_{r>0} \mint{Q^\lambda_r(x,t)}{|f(y,s)|}{dyds},
\end{equation*}
where $Q^\lambda_r(x,t)$ is the intrinsic parabolic cylinder as in \eqref{ipc}.
We also use
\begin{equation}\label{intrinsic mf}
\m^\lambda_U f := \m^\lambda \left(\chi_U f\right)
\end{equation}
if $f$ is defined on a set $U$ with the usual characteristic function $\chi_U$ of $U$. In particular, when $\lambda=1$ or $p=2$, it becomes the {\em classical maximal function} $\M f$.

We will use the following weak $(1,1)$-estimates for the $\la$-maximal function:
\begin{lemma}\label{WE}
If $f \in L^1(\mr^{n+1})$, then there exists a constant $c=c(n)\ge1$ such that
\begin{align}\label{lambda 1-1}
\left|\mgh{(y,s)\in \mr^{n+1}: \m^\lambda f (y,s) >\alpha}\right| \leq \frac{c}{\alpha}\integral{\bb^{n+1}}{|f(x,t)|}{dx dt}
\end{align}
for any $\alpha>0$. Moreover, we have
\begin{align}\label{lambda 1-1_2}
\left|\mgh{(y,s)\in \mr^{n+1}: \m^\lambda f (y,s) > 2\alpha}\right| \leq \frac{c}{\alpha}\int_{\mgh{(y,s)\in \mr^{n+1}: |f|>\alpha}} |f|\ dxdt
\end{align}
for any $\alpha>0$.
\end{lemma}

\begin{proof}
Let $f^\lambda(x,t)=f(x,\lambda^{2-p}t)$ and set $s=\lambda^{2-p}\tau.$ Then we find
$$\mint{Q_r(y,\tau)}{\left|f^\lambda(x,t)\right|}{dxdt} = \mint{Q^\lambda_r(y,s)}{\left|f(x,t)\right|}{dxdt}$$
for any $r>0$ and any $(y,\tau)\in\mr^{n+1}$, which implies
\begin{align*}
\left(\m f^\lambda\right)(y,\tau) = \left(\m^\lambda f\right)(y,s).
\end{align*}
From the weak type $(1,1)$-estimate for the classical maximal function (see for instance \cite[Chapter I, Theorem 1]{Ste93}), it follows
\begin{align*}
\left|\left\{(y,s)\in \mr^{n+1}: \m^\lambda f (y,s) > \alpha\right\}\right|  &  =  \lambda^{2-p} \left|\left\{(y,\tau)\in \mr^{n+1}: \m f^\lambda (y,\tau) >\alpha\right\}\right|\\
&   \leq       c \frac{\lambda^{2-p}}{\alpha} \integral{\bb^{n+1}}{\left|f^\lambda(x,t)\right|}{dxdt}\\
&   =        \frac{c}{\alpha} \integral{\bb^{n+1}}{\left|f(x,t)\right|}{dxdt},
\end{align*}
which shows \eqref{lambda 1-1}.

To prove \eqref{lambda 1-1_2}, let us consider the function $f^\alpha := \chi_{\{|f|>\alpha\}}|f|$. Then it is straightforward to check that for $(y,s) \in \bb^{n+1}$,
$$|f(y,s)|\leq f^\alpha(y,s)+\alpha,$$
and so
$$\m^\lambda f (y,s) \leq \m^\lambda f^\alpha (y,s) + \alpha.$$
Therefore we see
$$\{(y,s)\in \mr^{n+1}: \m^\lambda f (y,s) > 2\alpha\} \subset \{(y,s)\in \mr^{n+1}: \m^\lambda f^\alpha (y,s) > \alpha\}.$$
Applying \eqref{lambda 1-1} to $\{(y,s)\in \mr^{n+1}:  \m^\lambda f^\alpha (y,s) > \alpha\}$, we derive \eqref{lambda 1-1_2}.
\end{proof}

%For fixed $\lambda>0$, we also obtain the strong $(p,p)$-estimate for the $\la$-maximal function as follows:
%\begin{lemma}
%\label{Mod_E}
%If $f \in L^q(\mr^{n+1})$ for some $q>1$, then we have
%\begin{align}\label{lambda p-p}
%\left\|\m^\lambda f \right\|_{L^q(\mr^{n+1})} \leq c \|f\|_{L^q(\mr^{n+1})}
%\end{align}
%for some constant $c=c(n,q)>0$.
%\end{lemma}

%\begin{proof}
%Let $f \in L^q(\mr^{n+1})$ for some $q>1$. We see from \eqref{lambda 1-1_2} and Fubini's theorem that
%\begin{align*}
%\int_{\bb^{n+1}} \left(\m^\lambda f\right)^q\; dxdt &= q 2^q \int_0^\infty \alpha^{q-1} \left|\left\{(y,\tau)\in \mr^{n+1}: \m^\lambda f (y,\tau) > 2\alpha\right\}\right|\; d\alpha \\
%&\leq cq 2^q \int_0^\infty \alpha^{q-2} \left[\int_{\{|f|>\alpha\}} \left|f(x,t)\right|\ dxdt \right]d\alpha \\
%& = cq 2^q \int_{\mr^{n+1}} \left[\int_0^{|f|} \alpha^{q-2}\; d\alpha\right] |f(x,t)|\; dxdt \\
%& = \frac{cq2^q}{q-1} \int_{\mr^{n+1}} |f(x,t)|^q\; dxdt,
%\end{align*}
%which completes the proof.
%\end{proof}

We record some useful property as follows:
\begin{lemma}\label{useful int}
Let $U$ be an open set in $\bb^{n+1}$. For any $q > l \ge 0$, we have
\begin{equation}\label{useful int1}
\integral{U}{|g|_k^{q-l} |g|^l}{dxdt} = (q-l) \int_0^k \la^{q-l-1}\bgh{\integral{\mgh{(y,s)\in U : |g(y,s)| > \la}}{|g|^l}{dxdt}} d\la
\end{equation}
for any $k>0$. Here the function $|g|_k := \min\mgh{|g|, k}$ is the truncation. If $g \in L^{q}(U)$, then \eqref{useful int1} also holds for $k=\infty$.
\end{lemma}

\begin{proof}
By Fubini's theorem, it is easy to check that Lemma \ref{useful int} holds.
\end{proof}

We end this section with the following modified version of Vitali's covering lemma for the intrinsic parabolic cylinder:
\begin{lemma}\label{VitCov}
Assume that $\Omega$ is $(\delta,R)$-Reifenberg flat. Let $0<\varepsilon<1$, let $\lambda>0$, and let $\C\subset \D$ be two bounded measurable subsets of $\omt := \OO \times (-\infty,T)$ such that the following two conditions hold:
\begin{enumerate}[(i)]
\item\label{vic1} $|\C|<\varepsilon \norm{Q^\lambda_{R/10}}$, and
\item\label{vic2} for any $(x,t) \in \omt$ and any $r \in \left(0,\frac{R}{10}\right]$ with $|\C\cap Q^\lambda_r(x,t)|\geq \varepsilon |Q^\lambda_r|$, $Q^\lambda_r(x,t)\cap \omt \subset \D$.
\end{enumerate}
Then we have
$$|\C|\leq \left(\frac{10}{1-\delta}\right)^{n+2}\varepsilon|\D| \leq \left(\frac{80}{7}\right)^{n+2}\varepsilon|\D|.$$
\end{lemma}

\begin{proof}
According to \eqref{vic1}, for almost every $(x,t) \in \C$, there is $r_{(x,t)} \in \gh{0,\frac{R}{10}}$ such that
\begin{equation}\label{vic-1}
\left| \C \cap \iq{r_{(x,t)}}(x,t) \right| = \ep \left|\iq{r_{(x,t)}} \right| \ \ \text{and} \ \ \left| \C \cap \iq{r}(x,t) \right| < \ep \left|\iq{r} \right|
\end{equation}
for all $r \in \left(r_{(x,t)}, \frac{R}{10} \right]$. Since $\mgh{\iq{r_{(x,t)}}(x,t) : (x,t) \in \C}$ is a covering of $\C$, the standard Vitali covering lemma (see  \cite[Theorem C.1]{Bog07} or \cite[Theorem 1.24]{EG15}) implies that there exists a disjoint subcovering $\mgh{\iq{r_i}(x_i,t_i) : (x_i,t_i) \in \C}_{i=1}^{\infty}$ such that
\begin{equation}\label{vic-2}
\C \subset \bigcup_{i\ge1} \iq{5r_i}(x_i,t_i) \ \ \text{and} \ \ |\C| \le 5^{n+2} \sum_{i\ge1}\left|\iq{r_i}\right|.
\end{equation}
On the other hand, we compute for $(x,t) \in \omt$ that
\begin{equation}\label{vic-2.5}
\begin{aligned}
\frac{\norm{\iq{r}(x,t)}}{\norm{\omt \cap \iq{r}(x,t)}} &= \frac{2\la^{2-p}r^2|B_r(x)|}{\norm{(-\infty,T)\cap\gh{t-\la^{2-p}r^2,t+\la^{2-p}r^2}}|\OO \cap B_r(x)|}\\
&= \frac{2\la^{2-p}r^2|B_r(x)|}{\norm{\gh{t-\la^{2-p}r^2,\min\{T,t+\la^{2-p}r^2\}}}|\OO \cap B_r(x)|}\\
&\le \frac{2|B_r(x)|}{|\OO \cap B_r(x)|}.
\end{aligned}
\end{equation}
Since $\Omega$ is $(\delta,R)$-Reifenberg flat, \eqref{vic-2.5} and \eqref{measure density} yield
\begin{equation}\label{vic-3}
\begin{aligned}
\sup_{0<r \le R} \sup_{(x,t) \in \omt} \frac{\norm{\iq{r}(x,t)}}{\norm{\omt \cap \iq{r}(x,t)}} \le \sup_{0<r \le R} \sup_{x \in \OO} \frac{2|B_r(x)|}{|\OO \cap B_r(x)|} \le 2\gh{\frac{2}{1-\delta}}^n.
\end{aligned}
\end{equation}
Finally, we have from \eqref{vic-2}, \eqref{vic-1}, \eqref{vic-3} and \eqref{vic2} that
\begin{equation*}
\begin{aligned}
|\C| &= \norm{\bigcup_{i\ge1} \gh{\C \cap \iq{5r_i}(x_i,t_i)}} \le \sum_{i\ge1} \norm{\C \cap \iq{5r_i}(x_i,t_i)} < \ep \sum_{i\ge1} \norm{\iq{5r_i}(x_i,t_i)}\\
&=\ep 5^{n+2} \sum_{i\ge1} \norm{\iq{r_i}(x_i,t_i)} \le \left(\frac{10}{1-\delta}\right)^{n+2} \ep \sum_{i\ge1} \norm{\iq{r_i}(x_i,t_i) \cap \omt}\\
&=\left(\frac{10}{1-\delta}\right)^{n+2} \ep \norm{\bigcup_{i\ge1} \gh{\iq{r_i}(x_i,t_i) \cap \omt}} \le \left(\frac{10}{1-\delta}\right)^{n+2}\ep |\D|,
\end{aligned}
\end{equation*}
which completes the proof.
\end{proof}

\begin{remark}\label{vicrmk1}
If we replace $\omt$ by $\OO_T$ in the lemma above, then we could not have the measure density condition \eqref{vic-3} independent of $\la$, see \eqref{vic-2.5}. For this reason we obtain comparison results on the localized region of $\omt$, not $\OO_T$, see Section \ref{intrinsic comparison} and \ref{la covering arguments} below.
\end{remark}

%\begin{remark}\label{vicrmk2}
%{\color{dred} \bf
%When considering weak solution of the quasilinear parabolic problems without measure data, We can extend $\OO_T$ to $\OO \times \bb$ by $t\ge T$, extend $f=0$ and $t<0, u \equiv 0$, see \cite{BOR13, Byun's papers}!!!!
%and then we only need to consider the lateral boundary. However, in case of measure data problems, since there is no uniqueness of SOLA in general, we cannot extend a SOLA $u$ of \eqref{pem1} for $t\ge T$, see also \eqref{ste1}!!!!!
%}
%\end{remark}

\begin{remark}\label{vicrmk3}
A shape of the parabolic intrinsic cylinders $\iq{r}(x,t)$ is essential to prove Lemma \ref{VitCov}; if we take the intrinsic cylinders having the top center such as $B_r(x) \times (t- \lambda^{2-p} r^2, t)$ instead of $\iq{r}(x,t)$, then the standard Vitali covering lemma no longer holds, see \cite[Theorem C.1]{Bog07}.
\end{remark}

%%%%%%%%%%%%%%%%%%%%%%%%%%%%%%%%%%%%%%%%%%%%%%%%%%%%%%%%
\section{Standard energy type estimates}\label{Standard energy type estimates}

In this section we derive standard energy type estimates for the parabolic measure data problem \eqref{pem1} as follows:
\begin{lemma}\label{standard estimate}
Let $u$ be a SOLA of \eqref{pem1} with $p > 2-\frac{1}{n+1}$ and let $0 < \kappa < p-\frac{n}{n+1}$. Then there exists a constant $c = c (n,\La_0,p,\kappa) \ge 1$ such that
\begin{equation}\label{ste1}
\gh{\mint{\OO_T}{|Du|^\kappa}{dx dt}}^{\frac{1}{\kappa}} \le c \bgh{\frac{|\mu|(\OO_T)}{|\OO_T|^{\frac{n+1}{n+2}}}}^{\frac{n+2}{(n+1)p-n}}. %\le c \integral{\OO_T}{\M_1(\mu)^{\frac{n+2}{(p-1)n+p}}}{dx dt}
\end{equation}
\end{lemma}

\begin{proof}
The ideas of our proof are developed in \cite[Lemma 4.1]{KM14b}. We also refer to \cite[Lemma 2.2]{BDGO97}, \cite[Lemma 4.3]{KM13b} and \cite[Lemma 4.1]{DM11}. A main idea of the proof is to take some truncations of a solution $u$, instead of $u$ itself, as a test function, see \eqref{st-3} and \eqref{st-12} below. For the convenience of the reader, we give all the technicalities of the proof.

{\em Step 1.} Since $u$ is a SOLA of \eqref{pem1},  there exists a sequence of weak solutions $\{u_m\}$ of the regularized problems
\eqref{sola eq} satisfying \eqref{sola1} and \eqref{sola3}. We will first show that
\begin{equation}\label{st-1}
\sup_{0<t<T} \integral{\OO \times \{t\}}{|u_m|}{dx} \le |\mu|(\OO_T)
\end{equation}
and
\begin{equation}\label{st-2}
\integral{\OO_T}{\frac{|Du_m|^p}{\gh{\alpha + |u_m|}^{\xi}}}{dx dt} \le c \frac{\alpha^{1-\xi}}{\xi-1} |\mu|(\OO_T)
\end{equation}
for $\alpha > 0$ and $\xi>1$, where $c = c(n,\La_0,p)\ge 1$. For any fixed $\varepsilon>0$, choose a test function
\begin{equation}\label{st-3}
\varphi_1(x,t) = \pm \min\mgh{1, \frac{\gh{[u_m]_h}_\pm(x,t)}{\ep}}\zeta(t),
\end{equation}
where $\zeta$ is a nonincreasing smooth function on $t\in \bb$ satisfying $0\le \zeta \le 1$ and $\zeta(t)=0$ for all $t\ge\tau$ with $\tau \in (0,T)$. Clearly, $\varphi_1(\cdot,t) \in W_0^{1,p}(\OO)$ with $\left|\varphi_1 \right| \le 1$.
Substituting $\varphi_1$ into the weak formulation \eqref{steklov weak sol} and integrating on $(0,T)$, we obtain
\begin{equation}\label{st-4}
\underbrace{\integral{\OO_T}{\partial_t [u_m]_h \varphi_1}{dx dt}}_{=: I_1} + \underbrace{\integral{\OO_T}{\vgh{\bgh{\mathbf{a}(Du_m,x,t)}_h, D\varphi_1}}{dx dt}}_{=: I_2} = \underbrace{\integral{\OO_T}{[\mu_m]_h\varphi_1}{dx dt}}_{=: I_3}.
\end{equation}
To estimate $I_1$, we see
\begin{equation}\label{st-5}
\partial_t [u_m]_h \min\mgh{1, \frac{\gh{[u_m]_h}_\pm(x,t)}{\ep}} = \pm \partial_t \integraL{0}{\gh{[u_m]_h}_\pm}{\min\mgh{1, \frac{s}{\ep}}}{ds},
\end{equation}
and then an integration by parts and Lemma \ref{steklov property} yield
\begin{equation}\label{st-6}
\begin{aligned}
I_1 &= \integral{\OO_T}{\bgh{\integraL{0}{\gh{[u_m]_h}_\pm}{\min\mgh{1, \frac{s}{\ep}}}{ds}}  \gh{-\zeta_t}}{dx dt}\\
&\qquad  - \integral{\OO}{\bgh{\integraL{0}{\gh{[u_m]_h}_\pm(x,0)}{\min\mgh{1, \frac{s}{\ep}}}{ds}}  \zeta(x,0)}{dx}\\
&\overset{h \to 0}{\longrightarrow} \integral{\OO_T}{\bgh{\integraL{0}{\gh{u_m}_\pm}{\min\mgh{1, \frac{s}{\ep}}}{ds}}  \gh{-\zeta_t}}{dx dt} \ge 0, % \overset{\varepsilon \to 0}{\longrightarrow} \integral{\OO_T}{\gh{u_m}_\pm \gh{-\zeta_t}}{dx dt} \ge 0
\end{aligned}
\end{equation}
since $\zeta_t \le 0$. On the other hand, it follows from Lemma \ref{steklov property}, \eqref{monotonicity} and \eqref{sola3} that
\begin{gather}
\label{st-7} I_2 \overset{h \to 0}{\longrightarrow} \frac{1}{\varepsilon} \integral{\OO_T}{\vgh{\mathbf{a}(Du_m,x,t), Du_m}\chi_{\{0<\gh{u_m}_\pm <\varepsilon \}} \zeta(t)}{dx dt}  \ge 0 % \quad (\text{only if} \ \ s=0) ??????.
\end{gather}
and
\begin{gather}
\label{st-8} |I_3| \le \integral{\OO_T}{\left|[\mu_m]_h \right|}{dx dt} \overset{h \to 0}{\longrightarrow} \integral{\OO_T}{|\mu_m|}{dx dt} \le |\mu|(\OO_T).
\end{gather}
Combining \eqref{st-4}--\eqref{st-8}, we find
\begin{equation*}\label{st-9}
\integral{\OO_T}{\bgh{\integraL{0}{\gh{u_m}_\pm}{\min\mgh{1, \frac{s}{\ep}}}{ds}}  \gh{-\zeta_t}}{dx dt} \le |\mu|(\OO_T),
\end{equation*}
and letting $\varepsilon \to 0$ we see from Lebesgue's dominated convergence theorem that
\begin{equation*}\label{st-10}
\integral{\OO_T}{|u_m| \gh{-\zeta_t}}{dx dt} \le |\mu|(\OO_T).
\end{equation*}
Approximating $\zeta$ by the mollification of the characteristic function $\chi_{(-\infty,\tau)}$ yields
\begin{equation*}\label{st-10.5}
\integral{\OO\times\{\tau\}}{|u_m|}{dx} \le |\mu|(\OO_T)
\end{equation*}
for every $\tau \in (0,T)$, which implies \eqref{st-1}.

It also follows from \eqref{st-4}--\eqref{st-8} that
\begin{equation}\label{st-11}
\sup_{\varepsilon >0}  \frac{1}{\varepsilon} \integral{\OO_T}{\vgh{\mathbf{a}(Du_m,x,t), Du_m}\chi_{\{0<\gh{u_m}_\pm <\varepsilon \}}}{dx dt} \le |\mu|(\OO_T).
\end{equation}
We now select the test function
\begin{equation}\label{st-12}
\varphi_2 := \frac{\varphi_1}{\gh{\alpha+\gh{[u_m]_h}_\pm}^{\xi-1}},
\end{equation}
where both constants $\xi >1$ and $\alpha >0$ are to be determined later. Substituting $\varphi_2$ into \eqref{steklov weak sol} and integrating over $(0,T)$, we get
\begin{equation}\label{st-13}
\integral{\OO_T}{\partial_t [u_m]_h \varphi_2}{dx dt} + \integral{\OO_T}{\vgh{\bgh{\mathbf{a}(Du_m,x,t)}_h, D\varphi_2}}{dx dt} = \integral{\OO_T}{[\mu_m]_h\varphi_2}{dx dt}.
\end{equation}
By calculations similar to those in \eqref{st-5}--\eqref{st-8}, we find
\begin{gather}
\label{st-15} \lim_{h\to0} \left|\integral{\OO_T}{[\mu_m]_h\varphi_2}{dx dt}\right| \le \alpha^{1-\xi} |\mu|(\OO_T)
\end{gather}
and
\begin{gather}
\label{st-14} \sup_{\varepsilon>0}\lim_{h\to0}\integral{\OO_T}{\partial_t [u_m]_h \varphi_2}{dx dt} \le \alpha^{1-\xi} |\mu|(\OO_T).
\end{gather}
For the second term on the left-hand side of \eqref{st-13}, we establish
\begin{equation}\label{st-16}
\begin{aligned}
&\integral{\OO_T}{\vgh{\bgh{\mathbf{a}(Du_m,x,t)}_h, D\varphi_2}}{dx dt}\\
&\quad = \integral{\OO_T}{\vgh{\bgh{\mathbf{a}(Du_m,x,t)}_h, D\varphi_1} \frac{1}{\gh{\alpha+\gh{[u_m]_h}_\pm}^{\xi-1}}}{dx dt}\\
&\qquad + (1-\xi) \integral{\OO_T}{\vgh{\bgh{\mathbf{a}(Du_m,x,t)}_h, D\gh{[u_m]_h}_\pm} \frac{\varphi_1}{\gh{\alpha+\gh{[u_m]_h}_\pm}^{\xi}}}{dx dt}\\
&\quad =: J_1 + J_2.
\end{aligned}
\end{equation}
It follows from Lemma \ref{steklov property}, \eqref{st-7} and \eqref{st-11} that
\begin{equation}\label{st-17}
\sup_{\varepsilon>0}\lim_{h\to0} J_1 \le \alpha^{1-\xi} |\mu|(\OO_T).
\end{equation}
To estimate $J_2$, letting $h\to0$ gives
\begin{equation}\label{st-18}
\lim_{h\to0} J_2 = (1-\xi) \integral{\OO_T}{\vgh{\mathbf{a}(Du_m,x,t), D\gh{u_m}_\pm} \frac{\min\mgh{1, \frac{\gh{u_m}_\pm}{\ep}}\zeta(t)}{\gh{\alpha+\gh{u_m}_\pm}^{\xi}}}{dx dt}.
\end{equation}
Combining \eqref{st-13}--\eqref{st-18} and \eqref{monotonicity}, we discover
\begin{equation*}\label{st-19}
\integral{\OO_T}{\frac{|Du_m|^p}{\gh{\alpha+|u_m|}^{\xi}}\min\mgh{1, \frac{|u_m|}{\ep}}}{dx dt} \le c\frac{\alpha^{1-\xi}}{\xi-1} |\mu|(\OO_T)
\end{equation*}
for some constant $c=c(n,\La_0,p)\ge1$. As $\ep\to0$, we obtain \eqref{st-2}.

{\em Step 2.} Now recalling that $0 < \kappa < p-\frac{n}{n+1}$, we set
\begin{equation}\label{st-20}
\xi := \frac{n+1}{n}(p-\kappa) \quad \text{and} \quad \al := \gh{\mint{\OO_T}{|u_m|^{\frac{n+1}{n}\kappa}}{dxdt}}^{\frac{n}{(n+1)\kappa}}.
\end{equation}
Clearly $\xi >1$ and we may assume that $\al>0$. (If $\al=0$, then $u_m=0$ and \eqref{ste1} holds with a SOLA $u$ replaced by $u_m$. When $m\to\infty$, the proof is done.)

We first assume that $1\le \kappa < p-\frac{n}{n+1}$. From Lemma \ref{parabolic embedding} (with $q=\kappa$ and $l=1$) and \eqref{st-1}, we discover
\begin{equation}\label{st-21}
\begin{aligned}
\alpha &\le c(n,\kappa) \bgh{\gh{\mint{\OO_T}{|Du_m|^\kappa}{dxdt}}\gh{\sup_{0<t<T}\integral{\OO\times\{t\}}{|u_m|}{dx}}^{\frac{\kappa}{n}}}^{\frac{n}{(n+1)\kappa}}\\
&\le c |\mu|(\OO_T)^{\frac{1}{n+1}} \gh{\mint{\OO_T}{|Du_m|^\kappa}{dxdt}}^{\frac{n}{(n+1)\kappa}}.
\end{aligned}
\end{equation}
Then by H\"older's inequality, \eqref{st-2}, \eqref{st-20} and \eqref{st-21}, we deduce
\begin{equation}\label{st-22}
\begin{aligned}
&\mint{\OO_T}{|Du_m|^\kappa}{dx dt}  = \mint{\OO_T}{\gh{\frac{|Du_m|^p}{\gh{\alpha + |u_m|}^\xi}}^{\frac{\kappa}{p}} \gh{\alpha + |u_m|}^{\frac{\xi\kappa}{p}}}{dx dt}\\
&\qquad \le \gh{\mint{\OO_T}{\frac{|Du_m|^p}{\gh{\alpha + |u_m|}^\xi}}{dx dt}}^{\frac{\kappa}{p}} \gh{\mint{\OO_T}{\gh{\alpha + |u_m|}^{\frac{\xi\kappa}{p-\kappa}}}{dx dt}}^{\frac{p-\kappa}{p}}\\
&\qquad \le \frac{c}{(\xi-1)^{\frac{\kappa}{p}}} \gh{\frac{|\mu|(\OO_T)}{|\OO_T|}\al^{1-\xi}}^{\frac{\kappa}{p}}  \alpha^{\frac{\xi\kappa}{p}}  \qquad \gh{\text{since}\ \ \frac{\xi\kappa}{p-\kappa} = \frac{(n+1)\ka}{n}}\\
&\qquad \le c \gh{\frac{|\mu|(\OO_T)^{\frac{n+2}{n+1}}}{|\OO_T|}}^{\frac{\kappa}{p}} \gh{\mint{\OO_T}{|Du_m|^\kappa}{dxdt}}^{\frac{n}{(n+1)p}}\\
\end{aligned}
\end{equation}
for some constant $c = c (n,\La_0,p,\kappa) \ge 1$. Applying Young's inequality to \eqref{st-22} and letting $m\to\infty$, we finally obtain the estimate \eqref{ste1} for $1\le \kappa < p-\frac{n}{n+1}$.

If $0<\kappa<1$, then it follows from H\"{o}lder's inequality that
$$
\gh{\mint{\OO_T}{|Du|^\kappa}{dx dt}}^{\frac{1}{\kappa}} \le \mint{\OO_T}{|Du|}{dx dt} \le c \bgh{\frac{|\mu|(\OO_T)}{|\OO_T|^{\frac{n+1}{n+2}}}}^{\frac{n+2}{(n+1)p-n}},
$$
which we have used the estimate \eqref{ste1} with $\kappa=1$ for the second inequality above. This completes the proof.
\end{proof}

\begin{remark}
We point out that the constant $c$ in \eqref{ste1} blows up when $\kappa \nearrow p-\frac{n}{n+1}$, since $c$ is proportional to $\frac{1}{\xi-1}$ and $\xi \searrow 1$ as $\kappa \nearrow p-\frac{n}{n+1}$, see \eqref{st-22} and \eqref{st-20}.
%{\color{dred} \bf
%Indeed, if $\mu \in L^1(0,T;L^1\log L^1(\OO))$, then the estimate \eqref{ste1} holds for $\kappa = p-\frac{n}{n+1}$, that is,
%\begin{equation}\label{stelog}
%\integral{\OO_T}{|Du|^{p-\frac{n}{n+1}}}{dxdt} \le c(n,\La_0,p) \bgh{|\mu|(\OO_T)}^{\frac{n+2}{n+1}},??????
%\end{equation}
%see \cite[Theorem 1.8]{BDGO97} for details. Since $|\mu|(\OO_T) = |\mu|(\bb^{n+1}) < \infty$, the above estimate \eqref{stelog} is satisfied when $T \to \infty$. However, the estimate \eqref{ste1} does not hold in the same situation (RHS of  \eqref{ste1} blows up when $T \to \infty$), and thus we cannot extend a SOLA $u$ to $T=\infty$, see also Remark \ref{vicrmk2}.????????
%}
\end{remark}

%%%%%%%%%%%%%%%%%%%%%%%%%%%%%%%%%%%%%%%%%%%%%%%%%%%%%%%%%%%%%%%%%%%%%%%%%%%%%%%%%%%%%

\section{$L^1$-comparison estimates on intrinsic parabolic cylinders}\label{intrinsic comparison}

In this section, when considering our problem \eqref{pem1} with \eqref{str1} and \eqref{p-range}, we first extend a solution $u$ by zero for $t<0$ and assume that
\begin{equation}\label{appro mu}\left\{
\begin{aligned}
&\mu \in L^1(\omt) \cap L^{p'}(-\infty,T;W^{-1,p'}(\OO)),\\
&\mu \equiv 0 \quad \text{for} \   t \le 0,
\end{aligned}\right.
\end{equation}
see Remark \ref{vicrmk1} for the discussion of the time extension.
We establish comparison $L^1$-estimates for the spatial gradient of the weak solution $u$ to \eqref{pem1} in a localized intrinsic parabolic cylinder near the boundary of $\OO$. We only treat comparison results in such a boundary region, as the counterparts in an interior region can be derived in the same way.

Suppose that $(\mathbf{a},\OO)$ is $(\delta,R)$-vanishing for some $R>0$, where $\delta \in \gh{0,\frac18}$ is to be determined later. Fix any $\la >0$, $(x_0,t_0) \in \omt$ and $0 < r \le \frac{R}{8}$ satisfying
\begin{equation*}\label{reifen}
B_{8r}^+(x_0) \subset B_{8r}(x_0) \cap \OO \subset B_{8r}(x_0) \cap \{x \in \bb^n : x_n > -16\delta r\}.
\end{equation*}
Recall from Section \ref{notation} that $\omt := \OO \times (-\infty,T)$ and $K_r^\lambda(x_0,t_0) := Q_r^\lambda(x_0,t_0) \cap \omt$. Throughout this section, for simplicity, we omit the center point $(x_0,t_0)$ in the intrinsic parabolic cylinder $\ik{r}(x_0,t_0)$. We denote, for a measurable set $E \subset \bb^{n+1}$,
\begin{equation}\label{regular measure}
|\mu|(E): = \integral{E}{|\mu(x,t)|}{dx dt}
\end{equation}
and write
\begin{equation*}
\chi_{\{p <2\}} := \left\{
\begin{alignedat}{2}
0 & \quad \text{if} \ \ p \ge 2,\\
1 & \quad \text{if} \ \ p < 2,
\end{alignedat}\right.
\quad \text{and} \quad
\tilde{d}:=\max\{1,p-1\}.
\end{equation*}

Let $w$ be the unique weak solution to the Cauchy-Dirichlet problem
\begin{equation}
\label{pem2}\left\{
\begin{alignedat}{3}
w_t -\ddiv \mathbf{a}(Dw,x,t) &= 0 &&\quad \text{in} \ \ik{8r}, \\
w &= u &&\quad\text{on} \ \partial_p \ik{8r}.
\end{alignedat}\right.
\end{equation}

Now, using the measure density condition \eqref{measure density}, we extend the comparison estimates from \cite[Lemma 4.1]{KM14b} and \cite[Lemma 4.3]{KM13b} up to the boundary, as follows:
\begin{lemma}\label{compa1}
Let $u$ be a weak solution of \eqref{pem1} and let $w$ as in \eqref{pem2}. % with \eqref{reifen}.
Then there exists a constant $ c=c(n,\La_0,p)\ge1$ such that
\begin{equation*}\label{compa1-r}
\begin{aligned}
\gh{\mint{\ik{8r}}{|Du-Dw|^{\tilde{d}}}{dxdt}}^{\frac{1}{\tilde{d}}} &\le c \bgh{ \frac{|\mu|(\ik{8r})}{|\ik{8r}|^{\frac{n+1}{n+2}}} }^{\frac{n+2}{(n+1)p-n}}\\
&\quad + c \chi_{\{p<2\}}\bgh{ \frac{|\mu|(\ik{8r})}{|\ik{8r}|^{\frac{n+1}{n+2}}} } \gh{\mint{\ik{8r}}{|Du|}{dxdt}}^{(2-p)\frac{n+1}{n+2}}.
\end{aligned}
\end{equation*}
\end{lemma}

We next derive a boundary higher integrability result of the spatial gradient of the weak solution $w$ to the problem {\eqref{pem2}. For the case of an interior higher integrability, we refer to \cite{KL00}.

\begin{lemma}\label{high int}
Let $w$ be the weak solution of \eqref{pem2}. Assume that
\begin{equation}\label{hi-c}
\mint{\ik{8r}}{|Dw|^{\tilde{d}}}{dxdt} \le c_w \lambda^{\tilde{d}}
\end{equation}
for some constant $c_w>0$. Then there exist $\sigma=\sigma(n,\La_0,\La_1,p)>0$ and $c=c(n,\La_0,\La_1,p,c_w)\ge1$ such that
\begin{equation}\label{hi-r}
\mint{\ik{4r}}{|Dw|^{p(1+\sigma)}}{dx dt} \le c\lambda^{p(1+\sigma)}.
\end{equation}
\end{lemma}

\begin{proof}
We first set
\begin{equation*}\label{hi-1}
\tilde{\OO} = \mgh{y \in \bb^n : ry \in \OO} \quad \text{and} \quad \tilde{T} = \frac{T}{\la^{2-p}r^2}
\end{equation*}
and consider the rescaled operator $\tilde{\mathbf{a}}=\tilde{\mathbf{a}}(\xi,x,t):
\bb^n \times \bb^n \times \bb \rightarrow \bb^n$ defined by
\begin{equation*}\label{hi-2}
\tilde{\ba}(\xi,x,t) = \frac{\ba\gh{\la\xi, rx, \la^{2-p}r^2 t}}{\la^{p-1}}.
\end{equation*}
We set
\begin{equation*}\label{hi-3}
\tilde{u}(x,t) = \frac{u\gh{rx, \la^{2-p}r^2 t}}{\la r} \quad \text{and} \quad \tilde{w}(x,t) = \frac{w\gh{rx, \la^{2-p}r^2 t}}{\la r}
\end{equation*}
for $(x,t) \in \tilde{K}_8 := Q_8 \cap \gh{\tilde{\OO} \times (-\infty,\tilde{T})}$. It is straightforward to check that $\tilde{\ba}$ satisfies the structure condition \eqref{str1} with $\ma$ replaced by $\tilde{\ma}$ and that $\tilde{w}$ is the weak solution to the following problem
\begin{equation*}\label{hi-4}
\left\{
\begin{alignedat}{3}
\tilde{w}_t -\ddiv \tilde{\ba}(D\tilde{w},x,t) &= 0 &&\quad \text{in} \ \tilde{K}_8, \\
\tilde{w} &= \tilde{u} &&\quad \text{on} \ \partial_p \tilde{K}_8.
\end{alignedat}\right.
\end{equation*}
Note that since $\OO$ is $(\delta,R)$-Reifenberg flat, the measure density condition \eqref{measure density} holds, which implies that $\bb^n \setminus \OO$ becomes uniformly $p$-thick, see \cite[Section 3.1]{BP10} for more details. Then we deduce from \cite[Theorem 2.2]{BP10} and \cite[Remark 6.12]{Giu03} that
\begin{equation}\label{hi-5}
\mint{\tilde{K}_4}{|D\tilde{w}|^{p(1+\sigma)}}{dxdt} \le c \gh{\mint{\tilde{K}_8}{|D\tilde{w}|^{ps}}{dxdt}}^{\frac{1+\mathfrak{d}\sigma}{1-\mathfrak{d}+\mathfrak{d}s}} + c
\end{equation}
for every $s \in \left(\frac{\mathfrak{d}-1}{\mathfrak{d}},1\right]$, where $c=c(n,\La_0,\La_1,p,s)\ge1$ and the constant $\mathfrak{d}$ is given by
\begin{equation*}
\mathfrak{d} := \left\{
\begin{alignedat}{3}
&\frac{p}{2} &&\quad \text{if} \ \  p\ge2, \\
&\frac{2p}{(n+2)p-2n} &&\quad \text{if} \ \  \frac{2n}{n+2}< p \le 2.
\end{alignedat}\right.
\end{equation*}
After taking $s$ such that
\begin{equation*}
s := \left\{
\begin{alignedat}{3}
&\frac{p-1}{p} &&\quad \text{if} \ \  p\ge2, \\
&\frac{1}{p} &&\quad \text{if} \ \  2-\frac{1}{n+1}< p \le 2,
\end{alignedat}\right.
\end{equation*}
we scale back in \eqref{hi-5} and employ \eqref{hi-c} to discover \eqref{hi-r}.
\end{proof}

We next consider a freezing operator $\bar{\ma}_{B_{4r}^+} = \bar{\ma}_{B_{4r}^+}(\xi,t) : \bb^n \times I_{4r}^\la \to \bb^n$ given by
\begin{equation*}
\bar{\mathbf{a}}_{B_{4r}^+}(\xi,t) := \mint{B_{4r}^+}{\mathbf{a}(\xi,x,t)}{dx},
\end{equation*}
where $I_{4r}^\la := \gh{- \lambda^{2-p} (4r)^2, \lambda^{2-p} (4r)^2}$.
Then $\bar{\ma}_{B_{4r}^+}$ satisfies the structure condition \eqref{str1} with $\ma(\xi,\cdot,t)$ replaced by $\bar{\ma}_{B_{4r}^+}(\xi,t)$. Recalling \eqref{2-small BMO} of Definition \ref{main assumption}, we also observe
\begin{equation}\label{refined small BMO}
\mint{Q_{4r}^{\la,+}}{\Theta(\mathbf{a},B_{4r}^+)(x,t)}{dxdt} \le 4 \mint{Q_{4r}^{\la}}{\Theta(\mathbf{a},B_{4r})(x,t)}{dxdt} \le 4\delta.
\end{equation}

Let $v$ be the unique weak solution to the coefficient frozen problem
\begin{equation}
\label{pem3}\left\{
\begin{alignedat}{3}
v_t -\ddiv \bar{\ma}_{B_{4r}^+}(Dv,t) &= 0 &&\quad \text{in} \ K_{4r}^{\lambda}, \\
v &= w &&\quad \text{on} \ \partial_p K_{4r}^{\lambda}.
\end{alignedat}\right.
\end{equation}

From Lemma \ref{high int}, \eqref{measure density}, \eqref{refined small BMO} and \cite[Lemma 3.10]{BOR13}, we derive the comparison result between \eqref{pem2} and \eqref{pem3} as follows:
\begin{lemma}\label{fcompa1}
Let $w$ be the weak solution of \eqref{pem2} satisfying \eqref{hi-c} and let $v$ as in \eqref{pem3}. Then there is a constant
$c=c(n,\La_0,\La_1,p)\ge1$ such that
\begin{equation*}\label{fcompa1_r}
\mint{\ik{4r}}{|Dw-Dv|^{p}}{dx dt} \le c \delta^{\sigma_1} \la^p,
\end{equation*}
where $\sigma_1$ is a positive constant depending only on $n$, $\La_0$, $\La_1$ and $p$.
\end{lemma}

In the interior region $\gh{\iq{4r} \subset \OO_T}$, the spatial gradient of the weak solution $v$ to \eqref{pem3} is locally bounded, see \cite{DF85, DF85b, DiB93}. However, for the boundary case $\gh{\iq{4r} \not\subset \OO_T}$, the $L^\infty$-norm of $Dv$ need not necessarily be bounded when the boundary of $\OO$ may be extremely irregular. For this reason, we consider a weak solution $\bar{v}$ to the following limiting problem near the flat boundary:
\begin{equation}
\label{pem4}\left\{
\begin{alignedat}{3}
\bar{v}_t -\ddiv \bar{\mathbf{a}}_{B_{4r}^{+}}(D\bar{v},t) &= 0 &&\quad \text{in} \ Q_{2r}^{\lambda,+}, \\
\bar{v} &= 0 &&\quad \text{on} \ T_{2r}^\la,
\end{alignedat}\right.
\end{equation}
where $Q_{2r}^{\lambda,+}$ and $T_{2r}^\la$ are given in Section \ref{notation}.

Then from \cite[Theorem 1.6]{Lie93} (with $G(\tau)=\tau^p$ and $\varphi=0$), we obtain the boundedness of the spatial gradient of $\bar{v}$ near the flat boundary, as follows:
\begin{lemma}\label{bdy Lip reg}
For any weak solution $\bar{v}$ of \eqref{pem4}, we have
\begin{equation*}
\label{fcompa2_r2}
\Norm{D\bar{v}}^p_{L^{\infty}(Q_{r}^{\lambda,+})} \le c \mint{Q_{2r}^{\lambda,+}}{|D\bar{v}|^p}{dx dt} +c\lambda^p
\end{equation*}
for some constant $c=c(n,\La_0,\La_1,p) \ge 1$.
\end{lemma}

We next have the following comparison estimate between \eqref{pem3} and \eqref{pem4}:
\begin{lemma}\label{fcompa2}
For any $\varepsilon \in (0,1)$, there is a small constant
$\delta=\delta(n,\La_0,\La_1,p,\varepsilon)>0$ such that if $v$ is
the weak solution of \eqref{pem3} satisfying
\begin{equation*}\label{fcompa2-c}
\mint{\ik{4r}}{|Dv|^p}{dxdt} \le c_v \lambda^p
\end{equation*}
for some given constant $c_v \ge1$, then there exists a weak solution $\bar{v}$ of \eqref{pem4} such that
\begin{equation}\label{fcompa2_r1}
\mint{\ik{2r}}{|Dv-D\bar{v}|^p}{dx dt} \le \varepsilon^p \la^p \quad \text{and} \quad \mint{\ik{2r}}{|D\bar{v}|^p}{dx dt} \le c \la^p
\end{equation}
for some constant $c=c(n,\La_0,\La_1,p,c_v)\ge 1$, where $\bar{v}$ is extended by zero from $\iqq{2r}$ to $\ik{2r}$.
\end{lemma}

\begin{proof}
The first estimate in \eqref{fcompa2_r1} can be derived from the compactness argument as in \cite[Lemma 3.8]{BOR13}. From this estimate, \eqref{measure density} and the triangle inequality, we directly obtain the second estimate in  \eqref{fcompa2_r1}.
\end{proof}

We finally combine the previous results to obtain the boundary comparison $L^1$-estimate.
\begin{lemma}\label{comparison}
For any $\varepsilon  \in (0,1)$, there is a small constant $\delta=\delta(n,\La_0,\La_1,p,\varepsilon) >0$ such that the following holds: if $u$, $w$ and $v$ are the weak solutions of \eqref{pem1}, \eqref{pem2} and \eqref{pem3}, respectively, satisfying
\begin{equation}\label{comparison-c}
\mint{\ik{8r}}{|Du|^{\tilde{d}}}{dxdt} \le \la^{\tilde{d}} \quad \text{and} \quad \frac{|\mu|(\ik{8r})}{r^{n+1}} \le \delta\la,
\end{equation}
then there exists a weak solution $\bar{v}$ of \eqref{pem4} such that
\begin{equation}\label{comparison-r}
\mint{K^\lambda_{r}}{|Du-D\bar{v}|^{\tilde{d}}}{dxdt} \leq \varepsilon \lambda^{\tilde{d}} \quad\mathrm{and}\quad \|D\bar{v}\|_{L^\infty(K^{\lambda}_r)} \leq c \lambda,
\end{equation}
for some constant $c=c(n,\La_0,\La_1,p)\ge 1$, where $\bar{v}$ is extended by zero from $\iqq{2r}$ to $\ik{2r}$. Here $\tilde{d}:=\max\{1,p-1\}$.
\end{lemma}

\begin{proof}
First of all, it follows from \eqref{comparison-c} and Lemma \ref{compa1} that
\begin{equation}\label{comparison-1}
\mint{\ik{8r}}{|Du-Dw|^{\tilde{d}}}{dxdt} \le c\delta^{\tilde{d} \min\mgh{1,\frac{n+2}{(n+1)p-n}}}\la^{\tilde{d}} \ \text{and} \ \mint{\ik{8r}}{|Dw|^{\tilde{d}}}{dxdt} \le c_w\la^{\tilde{d}}
\end{equation}
for some constant $c_w=c_w(n,\La_0,p)\ge1$.
Combining Lemma \ref{high int} and \ref{fcompa1} yields
\begin{equation}\label{comparison-2}
\begin{aligned}
\mint{\ik{4r}}{|Dv|^p}{dx dt} &\le 2^{p-1} \mgh{\mint{\ik{4r}}{|Dw-Dv|^{p}}{dx dt} + \mint{\ik{4r}}{|Dw|^{p}}{dx dt}}\\
&\le c \delta^{\sigma_1} \la^p +  c\la^p =: c_v\la^p.
\end{aligned}
\end{equation}
Applying Lemma \ref{fcompa2} with $\ep$ replaced by $\tilde{\ep}$, we find a weak solution $\bar{v}$ of \eqref{pem4} such that
\begin{equation}\label{comparison-3}
\gh{\mint{\ik{2r}}{|Dv-D\bar{v}|^{\tilde{d}}}{dx dt}}^{\frac{1}{\tilde{d}}} \le \gh{\mint{\ik{2r}}{|Dv-D\bar{v}|^p}{dx dt}}^{\frac1p} \le \tilde{\ep} \la \le \frac{\ep}{3} \la,
\end{equation}
by selecting $\tilde{\ep}$ with $0<\tilde{\ep} \le \frac{\ep}{3}$.
Then we employ \eqref{measure density} and \eqref{comparison-1}--\eqref{comparison-3} to estimate
\begin{equation*}
\begin{aligned}
\mint{K^\lambda_{r}}{|Du-D\bar{v}|^{\tilde{d}}}{dxdt} &\le 2^{\tilde{d}-1} \mint{K^\lambda_{r}}{|Du-Dw|^{\tilde{d}}+|Dw-Dv|^{\tilde{d}}+|Dv-D\bar{v}|^{\tilde{d}}}{dxdt}\\
&\le c\delta^{\tilde{d} \min\mgh{1,\frac{n+2}{(n+1)p-n}}}\la^{\tilde{d}} + c \delta^{\frac{\sigma_1 \tilde{d}}{p}} \la^{\tilde{d}} + c\gh{\frac{\ep}{3} \la}^{\tilde{d}}\\
&\le \ep\la^{\tilde{d}},
\end{aligned}
\end{equation*}
by choosing $\delta$ sufficiently small.

On the other hand, the second estimate in \eqref{comparison-r} follows from Lemma \ref{bdy Lip reg} and \ref{fcompa2}.
\end{proof}

\begin{remark}\label{comparison rmk}
The second assumption in \eqref{comparison-c} is equivalent to the statement
\begin{equation}\label{comparison rmk-1}
\bgh{\frac{|\mu|(\ik{8r})}{\la^{2-p} r^{n+1}}}^{\frac{1}{p-1}} \le \delta^{\frac{1}{p-1}} \la.
\end{equation}
%\begin{equation}
%\frac{|\mu|(\ik{8r})}{r^{n+1}} \le \delta\la   \iff   \bgh{\frac{|\mu|(\ik{8r})}{\la^{2-p} r^{n+1}}}^{\frac{1}{p-1}} \le \delta^{\frac{1}{p-1}} \la,
%\end{equation}
In view of \eqref{comparison-c}, we see that the value $\frac{|\mu|(\ik{8r})}{r^{n+1}}$ is related to the intrinsic fractional maximal function $\M_1^\la(\mu)$, see Section \ref{la covering arguments} below. The function $\M_1^\la(\mu)$ also involves the {\em intrinsic Riesz potential ${\bf I}^{\mu}_{1,\la}$}, see \cite{KM14c, KM13b}. On the other hand, in view of \eqref{comparison  rmk-1}, we discover that the value $\bgh{\frac{|\mu|(\ik{8r})}{\la^{2-p} r^{n+1}}}^{\frac{1}{p-1}}$ is connected with the {\em intrinsic Wolff potential ${\bf W}^{\mu}_{\la}$}, see \cite{KM14b}. %Indeed, the second bound with $\la=1$ is exactly as in the elliptic case, see \cite{DM11, DM10}. These smallness assumptions determine the form of our main estimates!!!!!!
\end{remark}

%%%%%%%%%%%%%%%%%%%%%%%%%%%%%%%%%%%%%%%%%%%%%%%%%%%%%%%%
\section{$\lambda$-covering arguments}\label{la covering arguments}

We now consider a SOLA $u$ of \eqref{pem1} and the corresponding weak solutions $u_m$ ($m \in \mathbb{N}$) of \eqref{sola eq}. Suppose that $(\mathbf{a},\OO)$ is $(\delta,R)$-vanishing. Regarding $\mu_m$ in \eqref{sola eq} as an approximation of $\mu$ in \eqref{pem1} via mollification backward in time and then using a suitable cutoff function in time, $\mu_m\in L^\infty(\omt)$ satisfies the properties \eqref{sola2}, \eqref{sola3} and \eqref{appro mu}; then one can apply all the results obtained in Section \ref{intrinsic comparison} to $u=u_m$ and $\mu=\mu_m$. Also, we denote by $w_m$, $v_m$ and $\bar{v}_m$, the corresponding weak solutions of \eqref{pem2}, \eqref{pem3} and \eqref{pem4}, respectively.

Fix any $\ep \in (0,1)$, we set
\begin{equation}
\label{55}
\lambda_0 := \bgh{\frac{|\mu|(\OO_T)}{|\OO_T|^{\frac{n+1}{n+2}}}}^{\beta\tilde{d}} \frac{|\OO_T|}{\varepsilon \left|Q_{R/10}\right|} + \bgh{\frac{|\mu|(\OO_T)}{\delta T^{\frac{n+1}{2}}}}^{d} +1,
\end{equation}
where $\beta:=\frac{n+2}{(n+1)p-n}$, $\tilde{d}:=\max\{1,p-1\}$ and
\begin{equation}\label{dc}
d := \left\{
\begin{alignedat}{2}
&1 &&\quad \text{if} \ \  p\ge2, \\
&\frac{2}{(n+1)p-2n} &&\quad \text{if} \ \ 2-\frac{1}{n+1}< p \le 2.
\end{alignedat}\right.
\end{equation}
Note that the constants $\beta$, $\tilde{d}$ and $d$ come from Lemma \ref{standard estimate}, \ref{DecayLem} and  \ref{difference of ratio}, respectively. Recall from Section \ref{notation} that $\omt := \OO \times (-\infty,T)$ and $K_r^\lambda(x,t) := Q_r^\lambda(x,t) \cap \omt$. We may assume, upon letting $u \equiv 0$ for $t < 0$, that a SOLA $u$ is defined in $\omt$. For any fixed $N>1$ and  $\la \ge \la_0$, we write
\begin{equation*}
\C := \left\{ (x,t)\in\omt: \m_{\omt}^\lambda|Du|^{\tilde{d}} (x,t)> \gh{\lambda N}^{\tilde{d}} \right\}
\end{equation*}
and
\begin{equation*}
\D := \left\{(x,t)\in\omt:\m_{\omt}^\lambda |Du|^{\tilde{d}}(x,t) > \lambda^{\tilde{d}} \right\}\cup\left\{(x,t)\in\omt:\m^\lambda_1(\mu)(x,t)>\delta \lambda \right\},
\end{equation*}
where $\m_{\omt}^\lambda$ and $\m^\lambda_1$ are given by \eqref{intrinsic mf} and \eqref{intrinsic fractional mf}, respectively.
Note that we infer from Lemma \ref{DecayLem}, \ref{difference of ratio} and \ref{difference of ratio2} that both the upper level sets $\C$ and $\D$ are bounded measurable subsets of $\omt$.

We now verify two assumptions of the covering lemma (Lemma \ref{VitCov}).
\begin{lemma}\label{DecayLem}
There exists  a constant $N_1=N_1(n,\La_0,p)>1$ such that for any fixed $N\ge N_1$ and $\la \ge \la_0$, we have
\begin{align*}
\left|\C\right| < \varepsilon\left|Q_{R/10}^\lambda\right|.
\end{align*}
\end{lemma}

\begin{proof}
Fix any $N \ge N_1$ and $\la \ge \la_0 >1$. We have from  \eqref{lambda 1-1} and Lemma \ref{standard estimate} (with $\kappa=\tilde{d}:=\max\mgh{1,p-1}$) that
\begin{align*}
|\C| \le \frac{c}{(\lambda N)^{\tilde{d}}} \int_{\omt} |Du|^{\tilde{d}}\; dxdt  = \frac{c}{(\lambda N)^{\tilde{d}}} \int_{\OO_T} |Du|^{\tilde{d}}\ dxdt  \leq \frac{c |\OO_T|}{\left(\lambda N\right)^{\tilde{d}}}\bgh{\frac{|\mu|(\OO_T)}{|\OO_T|^{\frac{n+1}{n+2}}}}^{\beta\tilde{d}}
\end{align*}
for some constant $c=c(n,\La_0,p)\ge1$.

If $p \ge 2$, then we compute from \eqref{55} that
\begin{equation*}\label{56}
\begin{aligned}
|\C| &\le \frac{c |\OO_T|}{\left(\lambda N\right)^{p-1}}\bgh{\frac{|\mu|(\OO_T)}{|\OO_T|^{\frac{n+1}{n+2}}}}^{\beta(p-1)} \leq \frac{c\lambda^{2-p} |\OO_T|}{N^{p-1}} \bgh{\frac{|\mu|(\OO_T)}{|\OO_T|^{\frac{n+1}{n+2}}}}^{\beta(p-1)} \lambda_0^{-1} \\
&< \frac{c\lambda^{2-p}\varepsilon \left|Q_{R/10}\right|}{N_1^{p-1}} \le 2\varepsilon \lambda^{2-p}\left|Q_{R/10}\right| = \varepsilon\left|Q_{R/10}^\lambda\right|,
\end{aligned}
\end{equation*}
by choosing $N_1$ sufficiently large.

If $2-\frac{1}{n+1}< p \le 2$, then we note that $\la^{-1} \le \la_0^{-1} \le \la^{2-p} \la_0^{-1}$. Therefore

\begin{equation*}\label{57}
\begin{aligned}
|\C| &\le \frac{c  |\OO_T|}{\lambda N}\bgh{\frac{|\mu|(\OO_T)}{|\OO_T|^{\frac{n+1}{n+2}}}}^{\beta} \leq \frac{c \lambda^{2-p}|\OO_T|}{N}\bgh{\frac{|\mu|(\OO_T)}{|\OO_T|^{\frac{n+1}{n+2}}}}^{\beta} \lambda_0^{-1} \\
&< \frac{c\lambda^{2-p}\varepsilon \left|Q_{R/10}\right|}{N_1} \le \varepsilon\left|Q_{R/10}^\lambda\right|,
\end{aligned}
\end{equation*}
by selecting $N_1$ large enough.
\end{proof}

%\begin{proof}
%The proof will be divided into two cases.
%
%$\mathbf{case\ 1 : }\  p>2$.
%
%We start with a direct calculation as follows:
%\begin{align*}
%\left|\left\{ (x,t)\in\omt: \m_{\omt}^\lambda|Du| (x,t)> \lambda N \right\}\right| &= \left|\left\{ (x,t)\in\omt: \left(\m_{\omt}^\lambda|Du|\right)^{p-1} (x,t)> (\lambda N)^{p-1} \right\}\right| \\
%&\leq \frac{1}{(\lambda N)^{p-1}} \int_{\OO_T} \left(\m_{\omt}^\lambda|Du|\right)^{p-1}\; dxdt.
%\end{align*}
%By Lemma~\ref{Mod_E} and \ref{standard estimate} with $\alpha=p-1$, we have
%\begin{align*}
%\frac{1}{(\lambda N)^{p-1}} \int_{\OO_T} \left(\m_{\omt}^\lambda|Du|\right)^{p-1}\; dxdt & \leq \frac{c}{(\lambda N)^{p-1}} \int_{\OO_T} |Du|^{p-1}\; dxdt \\
%& \leq \frac{c_1}{\left(\lambda N\right)^{p-1}}|\mu|(\OO_T)^{\beta(p-1)}
%\end{align*}
%where $c_1$ depends only on $n$, $\gamma$, $p$ and $\omt$. But since
%$$
%\frac{c_1|\mu|(\OO_T)^{\beta(p-1)}}{\left(\lambda N\right)^{p-1}} = \frac{c_1\lambda^{2-p}|\mu|(\OO_T)^{\beta(p-1)}}{N^{p-1}} \lambda^{-1}
%$$
%and $\lambda_0 \geq 1$, we get the following estimates
%\begin{equation}\label{56}
%\begin{aligned}
%\left|\left\{ (x,t)\in\omt: \m_{\omt}^\lambda|Du| (x,t)> \lambda N \right\}\right| &\leq \frac{c_1\lambda^{2-p}|\mu|(\OO_T)^{\beta(p-1)}}{N^{p-1}} \lambda_0^{-1} \\
%&< \frac{c_1\lambda^{2-p}|\mu|(\OO_T)^{\beta(p-1)}}{N^{p-1}}\cdot \frac{\varepsilon N^{p-1}|Q_R|}{c_1|\mu|(\OO_T)^{\beta(p-1)}} \\
%&= \varepsilon \lambda^{2-p}|Q_R| = \varepsilon\left|Q_R^\lambda\right|
%\end{aligned}
%\end{equation}
%for any $\lambda \geq \lambda_0$.
%
%

%\end{proof}

\begin{lemma}
\label{2_lemma}
There exists $N_2=N_2(n,\La_0,\La_1, p)>1$ so that for any $\varepsilon\in(0,1)$, there is $\delta=\delta(n,\La_0,\La_1,p,\varepsilon) \in \left(0, \frac{1}{8}\right)$ such that the following holds: for any fixed  $\lambda \ge \la_0$, $N \ge N_2$, $r\in \left(0,\frac{R}{10}\right]$ and $(y,s)\in \omt$ with
\begin{equation}
\label{51}
\left|\C \cap Q_r^\lambda(y,s)\right| \ge \varepsilon\left|Q_r^\lambda\right|,
\end{equation}
we have
\begin{equation*}
K_r^\lambda(y,s) \subset \D.
\end{equation*}
\end{lemma}

\begin{proof}
We assume to the contrary that $K_r^\lambda(y,s)  \not\subset \D$ and derive a contradiction. Suppose that there is a point $(\tilde{x},\tilde{t}) \in K_r^\lambda(y,s)$ such that for all $\rho>0$,
\begin{equation}
\label{52}
\frac{1}{\left|Q^\lambda_{\rho}\right|}\int_{K_{\rho}^\lambda(\tilde{x},\tilde{t})} |Du|^{\tilde{d}}\ dxdt \leq \lambda^{\tilde{d}} \quad \text{and} \quad \frac{|\mu|(K^\lambda_{\rho}(\tilde{x},\tilde{t}))}{\rho^{n+1}} \leq \delta\lambda.
\end{equation}
We recall from \eqref{parabolic closure} that $\pc{K_{\rho}^\lambda(y,s)} \subset K_{\rho+r}^\lambda(\tilde{x},\tilde{t})$ for any $\rho > r$. Then it follows from \eqref{measure density} and \eqref{52} that
\begin{align*}
\mint{K^\lambda_\rho(y,s)}{|Du|^{\tilde{d}}}{dxdt} &\leq 2\left(\frac{16}{7}\right)^n \frac{1}{|Q^\lambda_\rho|}\int_{K^\lambda_\rho(y,s)} |Du|^{\tilde{d}}\ dxdt \\
&\leq 2\left(\frac{16}{7}\right)^n \left(\frac{9}{8}\right)^{n+2}\frac{1}{|Q^\lambda_{\rho+r}|}\int_{K^\lambda_{\rho+r}(\tilde{x},\tilde{t})} |Du|^{\tilde{d}}\ dxdt \\
&\leq \left(\frac{18}{7}\right)^{n+2}\lambda^{\tilde{d}}
\end{align*}
for every $\rho\geq 8r$. Then the resulting inequality and \eqref{sola1} imply that for any $\tilde{\ep} \in (0,1)$, we have
\begin{equation*}\label{co2-1}
\begin{aligned}
\mint{K^\lambda_\rho(y,s)}{|Du_m|^{\tilde{d}}}{dxdt} &\le 2^{\tilde{d}-1} \gh{\mint{K^\lambda_\rho(y,s)}{|Du|^{\tilde{d}}}{dxdt} + \mint{K^\lambda_\rho(y,s)}{|Du-Du_m|^{\tilde{d}}}{dxdt}}\\
& \le 2^{\tilde{d}-1} \gh{ \left(\frac{18}{7}\right)^{n+2} + \tilde{\ep}} \la^{\tilde{d}} \le 2^{\tilde{d}}\left(\frac{18}{7}\right)^{n+2} \la^{\tilde{d}} =: c_2 \la^{\tilde{d}}
\end{aligned}
\end{equation*}
whenever $m$ is large enough. Likewise, we recall \eqref{regular measure} with $\mu$ replaced by $\mu_m$, and then combine \eqref{52} and \eqref{sola2} to deduce
\begin{equation*}\label{co2-2}
\begin{aligned}
\frac{|\mu_m|(K^\lambda_\rho(y,s))}{\rho^{n+1}} &\le \frac{|\mu|(\pc{K^\lambda_\rho(y,s)})}{\rho^{n+1}} + \frac{\tilde{\ep}}{\rho^{n+1}}\\
&\le \left(\frac{9}{8}\right)^{n+1}\frac{|\mu|(K^\lambda_{\rho+r}(\tilde{x},\tilde{t}))}{(\rho+r)^{n+1}} + \frac{\tilde{\ep}}{\rho^{n+1}} \le c_2\delta\lambda
\end{aligned}
\end{equation*}
for any $\rho \geq 8r$, by selecting $\tilde{\ep}$ sufficiently small. Now applying  Lemma~\ref{comparison} with $(x_0,t_0)$, $\la$, $r$ and $\ep$ replaced by $(y,s)$, $c_2 \la$, $4r$ and $\eta$, respectively, we can find $\delta=\delta(n,\La_0,\La_1,p,\eta) >0$ such that
\begin{equation}\label{co2-3}
\mint{K^\lambda_{4r}(y,s)}{|Du_m-D\bar{v}_m|^{\tilde{d}}}{dxdt} \leq   \eta (c_2 \lambda)^{\tilde{d}}  \  \ \mathrm{and}\ \  \|D\bar{v}_m\|_{L^\infty\gh{K^{\lambda}_{4r}(y,s)}} \leq c c_2 \lambda =: c_3 \la
\end{equation}
for some constant $c_3 = c_3(n,\La_0,\La_1, p)\ge1$.
Combining \eqref{co2-3} and \eqref{sola1} yields
\begin{equation}\label{co2-4}
\mint{K^\lambda_{4r}(y,s)}{|Du-D\bar{v}_m|^{\tilde{d}}}{dxdt} \le c_4 \eta \la^{\tilde{d}}.
\end{equation}

We next show that
\begin{equation}\label{53}
\begin{aligned}
&\left\{ (x,t)\in K_r^\lambda(y,s):\m_{\omt}^\lambda|Du|^{\tilde{d}} > \gh{N\lambda}^{\tilde{d}} \right\}\\
&\qquad \subset \left\{(x,t)\in K_r^\lambda(y,s):\m_{K_{4r}^\lambda(y,s)}^\lambda  |Du-D\bar{v}_m|^{\tilde{d}} >\lambda^{\tilde{d}} \right\}
\end{aligned}
\end{equation}
provided $N \ge N_2 :=  \max\left\{2^{n+2}, 2^{\tilde{d}-1}\gh{1+{c_3}^{\tilde{d}}}\right\}$. To do so, let
$$(\tilde{y},\tilde{s}) \in \left\{ (x,t)\in K_r^\lambda(y,s) : \m_{K_{4r}^\lambda(y,s)}^\lambda |Du-D\bar{v}_m|^{\tilde{d}}\leq \lambda^{\tilde{d}} \right\}.$$
Then for any $\tilde{r}>0$,
\begin{equation}
\label{54}
\frac{1}{\left|Q_{\tilde{r}}^\lambda\right|}\int_{K^\lambda_{\tilde{r}}(\tilde{y},\tilde{s})\cap K_{4r}^\lambda(y,s)} |Du-D\bar{v}_m|^{\tilde{d}}\; dxdt \leq \lambda^{\tilde{d}}.
\end{equation}
If $\tilde{r}\in(0,2r]$, then $K_{\tilde{r}}^\lambda(\tilde{y},\tilde{s})\subset K^\lambda_{3r}(y,s)$, and so we have from \eqref{54} and \eqref{co2-3} that
\begin{align*}
\frac{1}{\left|Q_{\tilde{r}}^\lambda\right|}\int_{K^\lambda_{\tilde{r}}(\tilde{y},\tilde{s})}
|Du|^{\tilde{d}}\ dxdt &\leq \frac{2^{\tilde{d}-1}}{\left|Q_{\tilde{r}}^\lambda\right|}\int_{K^\lambda_{\tilde{r}}(\tilde{y},\tilde{s})} |Du-D\bar{v}_m|^{\tilde{d}}+|D\bar{v}_m|^{\tilde{d}}\ dxdt\\
&\le 2^{\tilde{d}-1}\gh{1+  {c_3}^{\tilde{d}}} \lambda^{\tilde{d}}.
\end{align*}
On the other hand, if $\tilde{r}>2r$, then $K^\lambda_{\tilde{r}}(\tilde{y},\tilde{s})\subset K^\lambda_{\tilde{r}+r}(y,s)\subset K^\lambda_{2\tilde{r}}(\tilde{x},\tilde{t})$; so the first inequality of \eqref{52} implies
\begin{align*}
\frac{1}{\left|Q_{\tilde{r}}^\lambda\right|}\int_{K^\lambda_{\tilde{r}}(\tilde{y},\tilde{s})}
|Du|^{\tilde{d}}\ dxdt \leq \frac{1}{\left|Q_{\tilde{r}}^\lambda\right|} \int_{K^\lambda_{2\tilde{r}}(\tilde{x},\tilde{t})} |Du|^{\tilde{d}}\ dxdt
\leq 2^{n+2}\lambda^{\tilde{d}}.
\end{align*}
Taking $N_2 = \max\left\{2^{n+2}, 2^{\tilde{d}-1}\gh{1+{c_3}^{\tilde{d}}}\right\}$, we obtain
$$(\tilde{y},\tilde{s}) \in \left\{ (x,t)\in K_r^\lambda(y,s):\m_{\omt}^\lambda|Du|^{\tilde{d}}\le \gh{N\lambda}^{\tilde{d}} \right\},$$
which implies \eqref{53}.

We now employ \eqref{53}, \eqref{lambda 1-1} and \eqref{co2-4} to conclude
\begin{align*}
&\left|\left\{(x,t)\in K_r^\lambda(y,s): \m_{\omt}^\lambda|Du|^{\tilde{d}}>  \gh{N \lambda}^{\tilde{d}} \right\}\right|\\
&\qquad \leq \left|\left\{(x,t)\in K_r^\lambda(y,s):\m_{K_{4r}^\lambda(y,s)}^\lambda |Du-D\bar{v}_m|^{\tilde{d}} > \lambda^{\tilde{d}} \right\}\right|\\
&\qquad \leq  \frac{c}{\lambda^{\tilde{d}}} \int_{K_{4r}^\lambda(y,s)} |Du-D\bar{v}_m|^{\tilde{d}} \ dxdt \leq c c_4 \eta \left|Q_r^\lambda\right| < \ep \left|Q_r^\lambda\right|,
\end{align*}
by selecting $\eta$ that satisfies the last inequality above. This is a contradiction to \eqref{51}, which completes the proof.
\end{proof}

When taking $N=\max\mgh{N_1,N_2}$ from Lemma \ref{DecayLem} and \ref{2_lemma}, we can apply Lemma \ref{VitCov} to obtain the following decay estimate:
\begin{corollary}\label{power decay estimate}
Under the assumptions as in Lemma \ref{DecayLem} and \ref{2_lemma}, we have
\begin{align*}
&\left|\left\{ (x,t)\in\omt: \m_{\omt}^\lambda|Du|^{\tilde{d}} > \gh{\lambda N}^{\tilde{d}} \right\}\right|\\
& \quad  \leq \varepsilon_0 \left(\left|\left\{(x,t)\in\omt:\m_{\omt}^\lambda |Du|^{\tilde{d}} > \lambda^{\tilde{d}} \right\}\right|+\left|\left\{(x,t)\in\omt:\m^\lambda_1(\mu)>\delta \lambda \right\}\right|\right),
\end{align*}
where $\ep_0 := \gh{\frac{80}{7}}^{n+2} \ep$.
\end{corollary}

Recalling \eqref{fractional mf} and \eqref{intrinsic fractional mf}, we obtain the following relation:
\begin{lemma}\label{difference of ratio}
Let $\la \ge 1$. Then we have
\begin{equation}\label{inclusion1}
\left\{(x,t) \in \omt : \m^\lambda_1(\mu) > \delta \lambda \right\} \subset \left\{(x,t) \in \omt : \bgh{\m_1(\mu)}^d > \delta^d \lambda \right\},
\end{equation}
where the constant $d$ is given in \eqref{dc}.
\end{lemma}

\begin{proof}
Let $\la \ge 1$. If $p \geq 2$, then we have $\iq{r}(x,t) \subset Q_r(x,t)$. It follows from the definitions \eqref{fractional mf} and \eqref{intrinsic fractional mf} that $\m^\lambda_1 \mu (x,t) \leq \m_1 \mu(x,t)$ for all $(x,t) \in \mr^{n+1}$, which implies that \eqref{inclusion1} holds for $p \geq 2$.

Assume $2-\frac{1}{n+1} < p\le2$. If  $(y,s) \in \left\{(x,t) \in \omt : \m^\lambda_1(\mu) > \delta \lambda \right\}$, then we have
$$
\frac{|\mu|(Q_r^\lambda(y,s))}{r^{n+1}} > \delta \lambda
$$
for any $r>0$. Since $Q_r^\lambda(y,s) \subset Q_{\lambda^{\frac{2-p}{2}}r}(y,s)$, we have
$$
\frac{|\mu|(Q_r^\lambda(y,s))}{r^{n+1}} \le \frac{|\mu|\gh{Q_{\frac{2-p}{2}r}(y,s)}}{r^{n+1}} = \left[\frac{|\mu|\gh{Q_{\lambda^\frac{2-p}{2}r}(y,s)}}{\left(\lambda^\frac{2-p}{2}r\right)^{n+1}}\right] \lambda^\frac{(2-p)(n+1)}{2}.
$$
By setting $r_0=\lambda^\frac{2-p}{2}r$, we discover
$$
\frac{|\mu|(Q_{r_0}(y,s))}{r_0^{n+1}} > \delta \lambda^\frac{(n+1)p-2n}{2},
$$
which implies $\bgh{\m_1(\mu)}^d(y,s) > \delta^d \lambda$; that is, $(y,s) \in \left\{(x,t) \in \omt : \bgh{\m_1(\mu)}^d > \delta^d \lambda \right\}$. Assertion \eqref{inclusion1} follows.
\end{proof}

If the given measure $\mu$ in $\omt$ can be decomposed into a finite signed Radon measure $\mu_0$ in $\OO$ and a Lebesgue function $f$ in $(-\infty,T)$, we write $\mu = \mu_0 \otimes f$.
\begin{lemma}\label{difference of ratio1.5}
Let $\la \ge 1$. If $\mu=\mu_0\otimes f$, then we have
\begin{equation}\label{inclusion2}
\begin{aligned}
&\left\{(x,t) \in \omt : \m^\lambda_1(\mu) > \delta \lambda \right\}\\
&\qquad \subset \left\{(x,t) \in \omt : \bgh{2(\m_1(\mu_0))(\m f)}^{\frac{1}{p-1}} > \delta^{\frac{1}{p-1}} \lambda \right\},
\end{aligned}
\end{equation}
where $\m_1(\mu_0)$ is given by \eqref{fractional mf2} and the maximal function $\m f$ is defined by
\begin{equation}\label{standard mf}
\m f (t) := \sup_{r>0} \fint_{t-r}^{t+r} |f(s)| \ ds = \sup_{r>0} \frac{1}{2r}\int_{t-r}^{t+r} |f(s)| \ ds.
\end{equation}
\end{lemma}

\begin{proof}
Let $\la \ge 1$. If  $\mu=\mu_0\otimes f$, then it follows
$$
\frac{|\mu|(Q_r^\lambda(x,t))}{\lambda^{2-p}r^{n+1}} = 2\frac{|\mu_0|(B_r(x))}{r^{n-1}}\fint_{t-\lambda^{2-p}r^2}^{t+\lambda^{2-p}r^2} \left|f(s)\right|\ ds.
$$
Taking the supremum with respect to $r$, we find
$$
\lambda^{p-2} \M^\lambda_1(\mu)(x,t)  \leq 2 \M_1(\mu_0) (x) \m f(t).
$$
If $\M^\lambda_1(\mu) > \delta\la$, then we see that $2(\m_1(\mu_0))(\m f) > \delta \lambda^{p-1}$, which implies \eqref{inclusion2}.
\end{proof}

\begin{remark}
The difference between the standard parabolic cylinder $Q_r(x,t)$ and the intrinsic parabolic cylinder $\iq{r}(x,t)$ causes the appearance of the deficit constant $d$ as in \eqref{dc}. This constant determines the exponent in our main estimates of Theorem \ref{main theorem}. However if $\mu=\mu_0\otimes f$, then we can derive more natural estimates without the deficit constant $d$, see Theorem \ref{main theorem2} and Remark \ref{main rk4}.
\end{remark}

Recall from Section \ref{notation} that $\OO_T := \OO \times (0,T)$ and $\omtb := \OO \times (-T,T)$.
\begin{lemma}\label{difference of ratio2}
Let $\la \ge \la_0$, where $\la_0$ be given in \eqref{55}. Then we have
\begin{equation}\label{inclusion3}
\left\{(x,t) \in \omt : \bgh{\m_1(\mu)}^d > \delta^d \lambda \right\} \subset \omtb,
\end{equation}
where the constant $d$ is given in \eqref{dc}.
\end{lemma}

\begin{proof}
Suppose $(x_0,t_0) \in \left\{(x,t) \in \omt : \bgh{\m_1(\mu)}^d > \delta^d \lambda \right\}$. Recall that the measure $\mu$ vanishes outside the domain $\OO_T$. If $t_0 \le -T$, then we see
\begin{equation*}
\m_1(\mu)(x_0,t_0) = \sup_{r>\sqrt{T}} \frac{|\mu|(Q_r(x_0,t_0))}{r^{n+1}} \le \frac{|\mu|(\OO_T)}{T^{\frac{n+1}{2}}}.
\end{equation*}
Since $\bgh{\m_1(\mu)(x_0,t_0)}^d > \delta^d \lambda$, we have
\begin{equation*}
\bgh{\frac{|\mu|(\OO_T)}{T^{\frac{n+1}{2}}}}^d > \delta^d \lambda,
\end{equation*}
a contradiction to our assumption $\la \ge \la_0$. Hence $-T < t_0 <T$, and so the assertion \eqref{inclusion3} follows.
\end{proof}

\begin{remark}
The measure $\mu$ vanishes outside $\OO_T$, but $\m_1(\mu)$ need not be zero in $\bb^{n+1} \setminus \OO_T$. Thus, the boundedness of the set $\left\{(x,t) \in \omt : \bgh{\m_1(\mu)}^d > \delta^d \lambda \right\}$ is important to derive the estimates in Theorem \ref{main theorem}, see Section \ref{Global gradient estimates for parabolic measure data problems} below.

In the case that $\mu=\mu_0\otimes f$, on the other hand,  we do not need to know the boundedness of the set $\left\{(x,t) \in \omt : \bgh{2(\m_1(\mu_0))(\m f)}^{\frac{1}{p-1}} > \delta^{\frac{1}{p-1}} \lambda \right\}$, see \eqref{61} below. Therefore, in this case, we can drop the condition $\lambda_0 \ge \bgh{\frac{|\mu|(\OO_T)}{\delta T^{\frac{n+1}{2}}}}^d$ in \eqref{55}; that is, if $\mu=\mu_0\otimes f$, then we take
\begin{equation}\label{55-2}
\lambda_0 := \bgh{\frac{|\mu_0|(\OO)\|f\|_{L^1(0,T)}}{|\OO_T|^{\frac{n+1}{n+2}}}}^{\beta\tilde{d}} \frac{|\OO_T|}{\varepsilon \left|Q_{R/10}\right|} +1,
\end{equation}
where $\beta:=\frac{n+2}{(n+1)p-n}$ and $\tilde{d}:=\max\{1,p-1\}$.
\end{remark}

Combining Corollary \ref{power decay estimate}, Lemma \ref{difference of ratio}, \ref{difference of ratio1.5} and \ref{difference of ratio2}, we finally obtain
\begin{lemma}\label{covarg}
Let $N=\max\mgh{N_1,N_2}$ from Lemma \ref{DecayLem} and \ref{2_lemma}. Then for any $\varepsilon\in(0,1)$, there exists $\delta=\delta(n,\La_0,\La_1,p,\varepsilon) \in \left(0, \frac{1}{8}\right)$ such that if  $(\mathbf{a},\OO)$ is $(\delta,R)$-vanishing for some $R>0$, then for any SOLA $u$ of \eqref{pem1} and any $\lambda \ge \la_0$, we have
\begin{equation*}\label{decay1}
\begin{aligned}
\left|\left\{ (x,t)\in\omt: \m_{\omt}^\lambda|Du|^{\tilde{d}} > \gh{N\lambda}^{\tilde{d}} \right\}\right| &\leq \varepsilon_0 \left|\left\{(x,t)\in\omt:\m_{\omt}^\lambda|Du|^{\tilde{d}} > \lambda^{\tilde{d}} \right\}\right|\\
&\qquad  +\varepsilon_0 \left|\left\{(x,t)\in\omtb:\bgh{\m_1(\mu)}^d>\delta^d \lambda \right\}\right|,
\end{aligned}
\end{equation*}
where $\ep_0 := \gh{\frac{80}{7}}^{n+2} \ep$, $\omtb := \OO \times (-T, T)$, $\tilde{d}:=\max\{1,p-1\}$, and $\m_{\omt}^\lambda$, $\m_1(\mu)$, $d$ are given by \eqref{intrinsic mf}, \eqref{fractional mf}, \eqref{dc}, respectively.

Furthermore, if $\mu=\mu_0\otimes f$, then we have
\begin{equation*}\label{decay2}
\begin{aligned}
&\left|\left\{ (x,t)\in\omt: \m_{\omt}^\lambda|Du|^{\tilde{d}} > \gh{N\lambda}^{\tilde{d}} \right\}\right|\\
&\qquad \leq \varepsilon_0 \left|\left\{(x,t)\in\omt:\m_{\omt}^\lambda|Du|^{\tilde{d}} > \lambda^{\tilde{d}} \right\}\right|\\
&\qquad \qquad  +\varepsilon_0 \left|\left\{(x,t) \in \omt : \bgh{2(\m_1(\mu_0))(\m f)}^{\frac{1}{p-1}} > \delta^{\frac{1}{p-1}} \lambda \right\}\right|,
\end{aligned}
\end{equation*}
where $\m_1(\mu_0)$ and $\m f$ are given by \eqref{fractional mf2} and \eqref{standard mf}, respectively.
\end{lemma}

%%%%%%%%%%%%%%%%%%%%%%%%%%%%%%%%%%%%%%%%%%%%%%%%%%%%%%%%%%%%%%%%%%%%%%%%%%%%%%%%%%%
\section{Global Calder\'on-Zygmund type estimates for parabolic measure data problems}\label{Global gradient estimates for parabolic measure data problems}

We devote this section to proving Theorem \ref{main theorem} and \ref{main theorem2}. Recall from Section \ref{notation} that $\OO_T = \OO \times (0,T)$, $\omt = \OO \times (-\infty,T)$ and $\omtb = \OO \times (-T,T)$. Assume that $(\mathbf{a},\OO)$ is $(\delta,R)$-vanishing. Fix any $\ep \in (0,1)$. Let $N=N(n,\La_0,\La_1,p)>1$ be a given constant in Lemma \ref{covarg}. We suppose, upon letting $u \equiv 0$ for $t < 0$, that a SOLA $u$ of \eqref{pem1} is defined in $\omt$. We also assume that the measure $\mu$ is defined in $\bb^{n+1}$ by considering the zero extension to $\mathbb{R}^{n+1}$. If $\mu=\mu_0 \otimes f$, where $\mu_0$ is a finite signed Radon measure on $\OO$ and $f \in L^s(0,T)$ for some $s \ge 1$, then we extend both $\mu_0$ and $f$ to be $0$ outside $\OO$ and $(0,T)$, respectively.

We first derive the following decay estimates of integral type:
\begin{lemma}\label{de1}
Let $\tilde{d} < \kappa_0 < p-\frac{n}{n+1}$ and let $\la \ge \la_0$, where $\tilde{d}:=\max\{1,p-1\}$ and $\la_0$ is given in \eqref{55}. If $u$ is a SOLA of \eqref{pem1}, then there exists a constant $c=c(n,\La_0,\La_1,p,\kappa_0)\ge1$ such that
\begin{equation}\label{de1-r}
\begin{aligned}
&\int_{\left\{(x,t)\in\omt : |Du|>N\lambda\right\}} |Du|^{\kappa_0}\ dxdt \\
&\qquad \leq c \varepsilon \int_{\left\{(x,t)\in\omt : |Du|>\frac{\lambda}{2}\right\}} |Du|^{\kappa_0}\ dxdt\\
&\qquad\qquad +\frac{c \varepsilon}{\delta^{d \kappa_0}} \int_{\left\{(x,t)\in\omtb : \bgh{\m_1(\mu)}^d>\delta^d \lambda\right\}} \bgh{\m_1(\mu)}^{d \kappa_0}\ dxdt,
\end{aligned}
\end{equation}
where $\m_1(\mu)$ and the constant $d$ are given in \eqref{fractional mf} and \eqref{dc}, respectively.
\end{lemma}

\begin{proof}
We first recall from Lemma \ref{standard estimate} that $Du \in L^{\kappa_0}(\omt)$. Having in mind that
$$
\m_{\omt}^\tau|Du|(x,t)\geq |Du|(x,t)
$$
for any $\tau>0$ and $(x,t)\in \omt$, we compute
\begin{equation*}\label{600}
\begin{aligned}
&\int_{\{(x,t)\in\omt:|Du|>N\lambda\}} |Du|^{\kappa_0}\ dxdt \\
&\qquad = \kappa_0 N^{\kappa_0}\int_\lambda^\infty \tau^{\kappa_0-1}\left|\left\{(x,t)\in\omt:|Du|>N \tau\right\}\right|\ d\tau\\
&\qquad\qquad + (N\lambda)^{\kappa_0}\left|\left\{(x,t)\in\omt:|Du|>N \lambda\right\}\right| \\
&\qquad \leq   \kappa_0 N^{\kappa_0} \underbrace{\int_\lambda^\infty \tau^{\kappa_0-1}\left|\left\{(x,t)\in\omt:\m_{\omt}^\tau|Du|^{\tilde{d}}> \gh{N\tau}^{\tilde{d}} \right\}\right|\ d\tau}_{=:J_1} \\
&\qquad\qquad + (N\lambda)^{\kappa_0} \underbrace{\left|\left\{(x,t)\in\omt:\m_{\omt}^\la|Du|^{\tilde{d}}> \gh{N\lambda}^{\tilde{d}}\right\}\right|}_{=:J_2}.
\end{aligned}
\end{equation*}
Applying Lemma \ref{covarg} and  \eqref{lambda 1-1_2}, we deduce
\begin{equation*}
\begin{aligned}
J_1 %&= \int_\lambda^\infty \tau^{\kappa_0-1}\left|\left\{(x,t)\in\omt:\m_{\omt}^\tau|Du|> N\tau\right\}\right|\ d\tau \\
& \leq \varepsilon_0 \int_\lambda^\infty \tau^{\kappa_0-1}\left|\left\{(x,t)\in\omt:\m_{\omt}^\tau|Du|^{\tilde{d}}> \tau^{\tilde{d}}\right\}\right|\ d\tau \\
& \quad\quad + \varepsilon_0 \int_\lambda^\infty \tau^{\kappa_0-1}\left|\left\{(x,t)\in\omtb: \bgh{\m_1(\mu)}^d> \delta^d \tau\right\}\right|\ d\tau\\
& \leq c(n) \varepsilon_0  \int_\lambda^\infty \tau^{\kappa_0-1-\tilde{d}}\left[\int_{\mgh{(x,t)\in\omt:|Du|>\frac{\tau}{2}}}|Du|^{\tilde{d}}\ dxdt\right]  d\tau \\
& \quad\quad + \varepsilon_0 \int_\lambda^\infty \tau^{\kappa_0-1}\left|\left\{(x,t)\in\omtb : \bgh{\m_1(\mu)}^d > \delta^d \tau\right\}\right|\ d\tau,
\end{aligned}
\end{equation*}
where $\varepsilon_0=\gh{\frac{80}{7}}^{n+2}\varepsilon$. Furthermore we see from Fubini's theorem that
\begin{equation}\label{kappa0}
\begin{aligned}
&\int_\lambda^\infty \tau^{\kappa_0-1-\tilde{d}}\left[\int_{{\mgh{(x,t)\in\omt:|Du|>\frac{\tau}{2}}}}|Du|^{\tilde{d}}\ dxdt\right] d\tau\\
&\qquad = \int_{{\mgh{(x,t)\in\omt:|Du|>\frac{\la}{2}}}} \left[\int_\lambda^{2|Du|} \tau^{\kappa_0-1-\tilde{d}}\ d\tau\right] |Du|^{\tilde{d}}\ dxdt \\
&\qquad \leq \frac{2^{\kappa_0-\tilde{d}}}{\kappa_0-\tilde{d}} \int_{{\mgh{(x,t)\in\omt:|Du|>\frac{\la}{2}}}} |Du|^{\kappa_0}\ dxdt,
\end{aligned}
\end{equation}
and that
\begin{equation*}
\begin{aligned}
&\int_\lambda^\infty \tau^{\kappa_0-1}\left|\left\{(x,t)\in\omtb : \bgh{\m_1(\mu)}^d > \delta^d \tau\right\}\right|\ d\tau\\
&\qquad \le \frac{1}{\kappa_0 \delta^{d \kappa_0}}\int_{\mgh{(x,t)\in\omtb:\bgh{\m_1(\mu)}^d>\delta^d \lambda}} \bgh{\m_1(\mu)}^{d \kappa_0}\ dxdt.
\end{aligned}
\end{equation*}
On the other hand, it follows from H\"{o}lder's inequality, Lemma \ref{covarg} and  \eqref{lambda 1-1_2} that
\begin{equation*}
\begin{aligned}
J_2 %&= \left|\left\{ (x,t)\in\omt: \m_{\omt}^\lambda|Du| > N\lambda \right\}\right| \\
&\leq \varepsilon_0 \left|\left\{(x,t)\in\omt:\m_{\omt}^\lambda|Du|^{\tilde{d}} > \lambda^{\tilde{d}} \right\}\right| +\varepsilon_0 \left|\left\{(x,t)\in\omtb:\bgh{\m_1(\mu)}^d>\delta^d \lambda \right\}\right| \\
&\leq \frac{c\varepsilon_0}{\la^{\kappa_0}} \int_{{\mgh{(x,t)\in\omt:|Du|>\frac{\la}{2}}}} |Du|^{\kappa_0}\ dxdt \\
&\quad\quad + \frac{\varepsilon_0}{ \delta^{d \kappa_0}\la^{\kappa_0}} \int_{\mgh{(x,t)\in\omtb:\bgh{\m_1(\mu)}^d>\delta^d \lambda}} \bgh{\m_1(\mu)}^{d \kappa_0}\ dxdt.
\end{aligned}
\end{equation*}

Combining the inequalities above, we obtain the estimate \eqref{de1-r}.
\end{proof}

The following result may be proved in much the same way as Lemma \ref{de1}.
\begin{lemma}\label{de2}
Let $\tilde{d} < \kappa_0 < p-\frac{n}{n+1}$ and let $\la \ge \la_0$, where $\tilde{d}:=\max\{1,p-1\}$ and $\la_0$ is given in \eqref{55-2}. If $u$ is a SOLA of \eqref{pem1} and $\mu=\mu_0\otimes f$, then there exists a constant $c=c(n,\La_0,\La_1,p,\kappa_0)\ge1$ such that
\begin{equation}\label{de2-r}
\begin{aligned}
&\int_{\left\{(x,t)\in\omt : |Du|>N\lambda\right\}} |Du|^{\kappa_0}\ dxdt\\
&\quad \leq c \varepsilon \int_{\left\{(x,t)\in\omt : |Du|>\frac{\lambda}{2}\right\}} |Du|^{\kappa_0}\ dxdt \\
&\qquad +  \frac{c \varepsilon}{\delta^{\kappa_0}} \int_{\left\{(x,t)\in\omt : \left[2\left(\M_1(\mu_0)\right) (\m f)\right]^\frac{1}{p-1}>\delta \lambda\right\}} \left[\left(\M_1(\mu_0) \right)(\m f)\right]^\frac{\kappa_0}{p-1}\ dxdt,
\end{aligned}
\end{equation}
where the operators $\m_1(\mu_0)$ and $\m f$ are given in \eqref{fractional mf2} and \eqref{standard mf}, respectively.
\end{lemma}

\begin{remark}\label{de remark}
Both the constants $c$ in \eqref{de1-r} and \eqref{de2-r} blow up as $\ka_0 \searrow  \tilde{d}$, see \eqref{kappa0}.
\end{remark}

We need a technical assertion in the proof of Theorem \ref{main theorem}.
\begin{lemma}\label{de3}
For each $(x,t) \in \OO_T$ we have
\begin{equation}\label{de3-r}
\m_1(\mu)(x,-t) \le \m_1(\mu)(x,t).
\end{equation}
\end{lemma}

\begin{proof}
Let $(x,t) \in \OO_T$. From the definition \eqref{fractional mf}, for every $\epsilon>0$ there exists $r_0>0$ such that
\begin{equation*}
\m_1(\mu)(x,-t) < \frac{|\mu|(Q_{r_0}(x,-t))}{r_0^{n+1}} + \epsilon.
\end{equation*}
Since $Q_{r_0}(x,-t) \cap \OO_T \subset Q_{r_0}(x,t) \cap \OO_T$, we have
\begin{equation*}
|\mu|(Q_{r_0}(x,-t)) = |\mu|\gh{Q_{r_0}(x,-t) \cap \OO_T} \le |\mu|\gh{Q_{r_0}(x,t) \cap \OO_T},
\end{equation*}
which implies
\begin{equation*}
\m_1(\mu)(x,-t) < \frac{|\mu|\gh{Q_{r_0}(x,t) \cap \OO_T}}{r_0^{n+1}} + \epsilon \le \m_1(\mu)(x,t) +\epsilon.
\end{equation*}
Letting $\epsilon \to 0$, we obtain \eqref{de3-r}.
\end{proof}

%%%%%%%%%%%%%%%%%%%%%%%
\subsection{Proof of Theorem \ref{main theorem}}

%We are in a position now to prove Theorem \ref{main theorem} and \ref{main theorem2}.
\begin{proof}[Proof of Theorem \ref{main theorem}]
Let $q >\tilde{d}$ and let $\ka_0 = \ka_0(n,p,q)$ be an arbitrary constant with $\tilde{d}< \kappa_0 < \min\mgh{p-\frac{n}{n+1},q}$, where $\tilde{d}:=\max\mgh{1,p-1}$. Given $k \in \mathbb{R}$, we define the truncation of the function $|Du|$ to be
$$|Du|_k := \min\{|Du|,k\}.$$
From Lemma \ref{useful int}, we compute for any $k > N\la_0$
\begin{equation*}\label{58}
\begin{aligned}
&\int_{\omt} |Du|_k^{q-\kappa_0}|Du|^{\kappa_0}\ dxdt \\
&\qquad = (q-\ka_0) N^{q-\ka_0} \int_0^\frac{k}{N} \lambda^{q-\kappa_0-1}\left[\int_{\{(x,t)\in\omt : |Du|>N\lambda\}}|Du|^{\kappa_0}\ dxdt\right] \ d\lambda\\
&\qquad \leq c \underbrace{\int_0^{\lambda_0} \lambda^{q-\kappa_0-1}\left[\int_{\{(x,t)\in\omt : |Du|>N\lambda\}}|Du|^{\kappa_0}\ dxdt\right] \ d\lambda}_{=:P_1} \\
&\qquad  \quad\quad + c \underbrace{\int_{\lambda_0}^\frac{k}{N} \lambda^{q-\kappa_0-1}\left[\int_{\{(x,t)\in\omt : |Du|>N\lambda\}}|Du|^{\kappa_0}\ dxdt\right] \ d\lambda}_{=:P_2}
\end{aligned}
\end{equation*}
for some constant $c = c(n,\La_0,\La_1,p,q) \ge 1$, where $\la_0$ be given in \eqref{55}. It follows from Lemma \ref{standard estimate} that
\begin{equation}\label{j1}
\begin{aligned}
P_1 &\leq \int_0^{\lambda_0} \lambda^{q-\kappa_0-1}\ d\lambda \int_{\omt}|Du|^{\kappa_0}\ dxdt \\
&= \frac{\lambda_0^{q-\kappa_0}}{q-\ka_0} \int_{\OO_T}|Du|^{\kappa_0}\ dxdt \leq c \lambda_0^{q-\kappa_0} \bgh{|\mu|(\OO_T)}^{\be\ka_0}
\end{aligned}
\end{equation}
for some constant $c = c(n,\La_0,p,q,\OO_T) \ge 1$, where $\beta :=\frac{(n+2)}{(n+1)p-n}$. Furthermore we see from Lemma \ref{de1} that
\begin{align*}
P_2 &\leq c\varepsilon \underbrace{\int_{\lambda_0}^\frac{k}{N} \lambda^{q-\kappa_0-1}\left[\int_{\mgh{(x,t)\in\omt : |Du|>\frac{\lambda}{2}}}|Du|^{\kappa_0}\ dxdt\right] \ d\lambda}_{=:S_1} \\
&\quad\quad + \frac{c\varepsilon}{\delta^{d \kappa_0}} \underbrace{\int_{\lambda_0}^\frac{k}{N} \lambda^{q-\kappa_0-1}\left[\int_{\mgh{(x,t)\in\omtb : \bgh{\m_1(\mu)}^d>\delta^d\lambda}}\bgh{\m_1(\mu)}^{d \kappa_0}\ dxdt\right] \ d\lambda}_{=:S_2},
\end{align*}
where $\m_1(\mu)$ and $d$ are given in \eqref{fractional mf} and \eqref{dc}, respectively. Since $N>1$ and $|2|Du||_{\frac{k}{N}} \leq 2|Du|_k$, we have
\begin{equation*}\label{s1}
\begin{aligned}
S_1 %&=  \int_{\lambda_0}^\frac{k}{N} \lambda^{q-\kappa_0-1}\left[\int_{\mgh{(x,t)\in\omt : |Du|>\frac{\lambda}{2}}}|Du|^{\kappa_0}\ dxdt\right] \ d\lambda \\
&\leq c \int_{\omt} \left[\int_0^{|2|Du||_{\frac{k}{N}}} \lambda^{q-\kappa_0-1}\ d\lambda\right] |Du|^{\kappa_0}\ dxdt \\
&\leq c \int_{\omt} \left[\int_0^{2|Du|_k} \lambda^{q-\kappa_0-1}\ d\lambda\right] |Du|^{\kappa_0}\ dxdt \\
&\leq c \int_{\omt} |Du|_k^{q-\kappa_0}|Du|^{\kappa_0}\ dxdt,
\end{aligned}
\end{equation*}
and similarly,
\begin{equation}\label{s2}
\begin{aligned}
S_2 %&=  \int_{\lambda_0}^\frac{k}{N} \lambda^{q-\kappa_0-1}\left[\int_{\mgh{(x,t)\in\omtb : \bgh{\m_1(\mu)}^d>\delta^d\lambda}}\bgh{\m_1(\mu)}^{d \kappa_0}\ dxdt\right] \ d\lambda \\
& \leq \int_0^\infty \lambda^{q-\kappa_0-1}\left[\int_{\mgh{(x,t)\in\omtb : \bgh{\m_1(\mu)}^d>\delta^d\lambda}}\bgh{\m_1(\mu)}^{d \kappa_0}\ dxdt\right] \ d\lambda \\
&\leq \int_{\omtb} \left[\int_0^{\left(\frac{\m_1(\mu)}{\delta}\right)^d} \lambda^{q-\kappa_0-1}\ d\lambda\right] \bgh{\m_1(\mu)}^{d \kappa_0}\ dxdt \\
&\leq \frac{c}{\delta^{d(q-\kappa_0)}} \int_{\omtb} \bgh{\m_1(\mu)}^{dq}\ dxdt.
\end{aligned}
\end{equation}
Therefore,
\begin{equation}\label{60}
P_2 \leq c \varepsilon \mgh{ \int_{\omt} |Du|_k^{q-\kappa_0}|Du|^{\kappa_0}\; dxdt + \frac{1}{\delta^{dq}}\int_{\omtb} \bgh{\m_1(\mu)}^{dq}\ dxdt }
\end{equation}
for some constant $c = c(n,\La_0,\La_1,p,q) \ge 1$.

We employ \eqref{j1} and \eqref{60} to derive that
\begin{align*}
\int_{\omt} |Du|_k^{q-\kappa_0}|Du|^{\kappa_0}\ dxdt &\leq c_0\varepsilon\int_{\omt} |Du|_k^{q-\kappa_0}|Du|^{\kappa_0}\ dxdt\\
&\qquad +  \frac{c\ep}{\delta^{dq}}\int_{\omtb} \bgh{\m_1(\mu)}^{dq}\ dxdt + c\lambda_0^{q-\kappa_0} \bgh{|\mu|(\OO_T)}^{\be\ka_0} \\
\end{align*}
for some $c_0 = c_0(n,\La_0,\La_1,p,q) \ge 1$. Now, we take $\varepsilon>0$ so small that
$c_0\varepsilon<1$, and then we can determine a corresponding $\delta=\delta(n,\La_0,\La_1,p,q)>0$. Letting $k\to\infty$, we obtain
\begin{equation}\label{mr1-1}
\begin{aligned}
\int_{\OO_T} |Du|^q\ dxdt &= \int_{\omt} |Du|^q\ dxdt\\
&\leq c \int_{\omtb} \bgh{\m_1(\mu)}^{dq}\ dxdt + c\lambda_0^{q-\kappa_0} \bgh{|\mu|(\OO_T)}^{\be\ka_0},
\end{aligned}
\end{equation}
where the equality above have used the fact that $u\equiv 0$ for $t\le0$.

Moreover recalling \eqref{55}, we deduce
\begin{equation*}
\lambda_0 \le c \left\{
\begin{alignedat}{2}
&\bgh{|\mu|(\OO_T)}^{\beta(p-1)} +1 &&\quad \text{if} \ \ p\ge2, \\
&\bgh{|\mu|(\OO_T)}^{\frac{2}{(n+1)p-2n}} +1 &&\quad \text{if} \ \ 2-\frac{1}{n+1}< p \le 2
\end{alignedat}\right.
\end{equation*}
for some constant $c = c(n,\La_0,\La_1,p,q,R,\OO_T) \ge 1$, where we have used the fact that $\be(p-1) \ge 1$ for $p\ge2$, and that $\be < \frac{2}{(n+1)p-2n}$ for $2-\frac{1}{n+1}< p \le 2$. Consequently,
\begin{equation}\label{la0}
\lambda_0^{q-\kappa_0} \bgh{|\mu|(\OO_T)}^{\be\ka_0} \le c \left\{
\begin{alignedat}{2}
&\bgh{|\mu|(\OO_T)}^{\beta(p-1)q} +1 &&\quad \text{if} \ \ p\ge2, \\
&\bgh{|\mu|(\OO_T)}^{\frac{2q}{(n+1)p-2n}} +1 &&\quad \text{if} \ \ 2-\frac{1}{n+1}< p \le 2.
\end{alignedat}\right.
\end{equation}
On the other hand, Lemma \ref{de3} implies
\begin{equation}\label{mr1-2}
\int_{\omtb} \bgh{\m_1(\mu)}^{dq}\ dxdt \le 2 \int_{\OO_T} \bgh{\m_1(\mu)}^{dq}\ dxdt.
\end{equation}

%Combining \eqref{dc}, \eqref{mr1-1}, \eqref{la0} and \eqref{mr1-2}, we deduce for $p\ge2$ that
%\begin{equation*}\label{mr1-3}
%\int_{\OO_T} |Du|^q\ dxdt \le c \mgh{\int_{\OO_T} \bgh{\m_1(\mu)}^{q}\ dxdt + \bgh{|\mu|(\OO_T)}^{\frac{(n+2)(p-1)q}{(n+1)p-n}} +1}
%\end{equation*}
%for some $c = c(n,\La_0,\La_1,p,q,R,\OO_T) \ge 1$. Similarly, we find for $2-\frac{1}{n+1}< p \le 2$ that
%\begin{equation*}\label{mr1-4}
%\int_{\OO_T} |Du|^q\ dxdt \le c \mgh{\int_{\OO_T} \bgh{\m_1(\mu)}^{\frac{2q}{(n+1)p-2n}}\ dxdt + \bgh{|\mu|(\OO_T)}^{\frac{2q}{(n+1)p-2n}} +1}.
%\end{equation*}
%Note that $|\mu|(\OO_T) \le diam(\OO_T)^{n+1} \M_1(\mu)$, which implies
%\begin{equation*}
%\bgh{|\mu|(\OO_T)}^{\alpha q} \le c(n,\alpha,q,\OO_T) \gh{\integral{\OO_T}{\M_1(\mu)}{dxdt}}^{\alpha q}
%\end{equation*}
%for any $\alpha>0$. In light of the above estimates, we finally obtain the estimate \eqref{main r1}, which completes the proof.
Combining \eqref{mr1-1}--\eqref{mr1-2}, we finally obtain the desired estimates \eqref{main r1-1} and \eqref{main r1-2}. This completes the proof.
\end{proof}

%%%%%%%%%%%%%%%%%%%%%%%
\subsection{Proof of Theorem \ref{main theorem2}}

%We are ready to prove Theorem \ref{main theorem2}.

\begin{proof}[Proof of Theorem \ref{main theorem2}]
Let $q >\tilde{d}$ and let $\ka_0= \ka_0(n,p,q)$ be an arbitrary constant with $\tilde{d} < \kappa_0 < \min\mgh{p-\frac{n}{n+1},q}$, where $\tilde{d}:=\max\mgh{1,p-1}$. Proceeding as in the proof of Theorem \ref{main theorem} and using Lemma \ref{de2}, we infer
\begin{equation}\label{mr2-1}
\begin{aligned}
\int_{\omt} |Du|_k^{q-\kappa_0}|Du|^{\kappa_0}\ dxdt &\le c \ep \int_{\omt} |Du|_k^{q-\kappa_0}|Du|^{\kappa_0}\ dxdt\\
&\qquad + c \lambda_0^{q-\kappa_0} \bgh{|\mu_0|(\OO)\|f\|_{L^1(0,T)}}^{\be\ka_0} + \frac{c\varepsilon}{\delta^{\kappa_0}} S.
\end{aligned}
\end{equation}
Here $|Du|_k := \min\{|Du|,k\}$, $\beta :=\frac{(n+2)}{(n+1)p-n}$, $\la_0$ is given in \eqref{55-2} and
\begin{equation*}
S := \int_{\lambda_0}^\frac{k}{N} \lambda^{q-\kappa_0-1}\left[\int_{\omt \cap \mgh{\left[2(\M_1(\mu_0))(\m f )\right]^\frac{1}{p-1} > \delta \lambda}} \left[(\M_1(\mu_0))(\m f )\right]^\frac{\kappa_0}{p-1}\ dxdt\right]  d\lambda,
\end{equation*}
where $\m_1(\mu_0)$ and $\m f$ are given in \eqref{fractional mf2} and \eqref{standard mf}, respectively.
By calculation as in \eqref{s2}, we similarly derive
\begin{equation}\label{s3}
\begin{aligned}
S% &=  \int_{\lambda_0}^\frac{k}{N} \lambda^{q-\kappa_0-1}\left[\int_{\{\left[(\M_1(\mu_0))(\m f )\right]^\frac{1}{p-1} > \delta \lambda\}} \left[(\M_1(\mu_0))(\m f )\right]^\frac{\kappa_0}{p-1}\ dxdt\right] \ d\lambda \\
%& \leq \int_0^\infty \lambda^{q-\kappa_0-1}\left[\int_{\{\left[(\M_1(\mu_0))(\m f )\right]^\frac{1}{p-1} > \delta \lambda\}} \left[(\M_1(\mu_0))(\m f )\right]^\frac{\kappa_0}{p-1}\ dxdt\right] \ d\lambda \\
&\leq \int_{\omt} \left[\int_0^{\frac{1}{\delta}\left[2(\M_1(\mu_0))(\m f )\right]^\frac{1}{p-1}} \lambda^{q-\kappa_0-1}\ d\lambda\right] \left[(\M_1(\mu_0))(\m f )\right]^\frac{\kappa_0}{p-1}\ dxdt \\
&\leq \frac{c}{\delta^{q-\kappa_0}} \int_{\omt} \left[(\M_1(\mu_0))(\m f )\right]^\frac{q}{p-1}\ dxdt \\
&= \frac{c}{\delta^{q-\kappa_0}} \int_{\OO} \bgh{\M_1(\mu_0)}^\frac{q}{p-1} dx \int_{-\infty}^T \left(\m f\right)^\frac{q}{p-1} dt.
\end{aligned}
\end{equation}
Applying the strong $\gh{\frac{q}{p-1},\frac{q}{p-1}}$-estimate for the function $f$ (see for instance \cite[Chapter I, Theorem 1]{Ste93}), we find
\begin{equation}\label{61}
\int_{-\infty}^T \left(\m f\right)^\frac{q}{p-1} \ dt \le c(n,p,q) \int_{-\infty}^T |f|^\frac{q}{p-1} \ dt = c \int_{0}^T |f|^\frac{q}{p-1} \ dt,
\end{equation}
where we have used the condition that $q>p-1$.

We employ \eqref{mr2-1}--\eqref{61} and then select $\ep>0$ sufficiently small (with $\delta=\delta(n,\La_0,\La_1,p,q)>0$ being determined), to discover
\begin{equation}\label{mr2-2}
\begin{aligned}
\int_{\omt} |Du|_k^{q-\kappa_0}|Du|^{\kappa_0}\ dxdt &\le c \int_{\OO_T} \left[(\M_1(\mu_0))f\right]^\frac{q}{p-1}\ dxdt\\
&\qquad + c \lambda_0^{q-\kappa_0} \bgh{|\mu_0|(\OO)\|f\|_{L^1(0,T)}}^{\be\ka_0}.
\end{aligned}
\end{equation}
On the other hand, \eqref{55-2} implies
\begin{equation}\label{la0-2}
\lambda_0^{q-\kappa_0} \bgh{|\mu_0|(\OO)\|f\|_{L^1(0,T)}}^{\be\ka_0} \le c \mgh{\bgh{|\mu_0|(\OO)\|f\|_{L^1(0,T)}}^{\beta\tilde{d}q} +1}.
\end{equation}

Finally, combining \eqref{mr2-2}--\eqref{la0-2} and letting $k\to\infty$, we obtain the desired estimate \eqref{main r2}. This completes the proof.
\end{proof}

%\begin{remark}
%$$
%\tilde{d}\geq \beta
%$$
%therefore
%$$
%\int_\Omega |Du|^q\ dxdt \leq c\left(\int_{\omt} \bgh{\m_1(\mu)}^{dq}\ dxdt + |\mu|(\OO_T)^{d_0 q}+1 \right).
%$$
%and
%$$
%\int_{\omt} |Du|_k^{q-\kappa_0}|Du|^{\kappa_0}\ dxdt \leq c\left(\int_{\omt}\left[(\M_1(\mu_0))f\right]^\frac{q}{p-1}\ dxdt + |\mu|(\OO_T)^{q\tilde{d}}+1 \right).
%$$
%\end{remark}

%%%%%%%%%%%%%%%%%%%%%%%%%%%%%%%%%%%%%%%%%%%%%%%%%%%%%%%%%%%%%%%%%%%%%%%%%%%%%%%%%%%%%%%%%%%%

\section*{Acknowledgments}
The authors thank an anonymous referee for valuable comments. The authors also thank Wontae Kim for a helpful comment on Lemma \ref{high int}.

%%%%%%%%%%%%%%%%%%%%%%%%%%%%%%%%%%%%%%%%%%%%%%%%%%%%%%%%%%%%%%%%%%%%%%%%%%%%%%%%%%%%%%%%%%%%


\end{document}